\pdfoutput=1
\documentclass[english,11pt]{article}
\usepackage[T1]{fontenc}
\usepackage[latin9]{inputenc}
\usepackage[top=1in, right=1in, left=1in, bottom=1in]{geometry}
\usepackage{amsmath}
\usepackage{amsthm}
\usepackage{amssymb}
\usepackage{mathtools}
\usepackage{etextools}
\usepackage[numbers,square]{natbib}

\usepackage{algorithmic}
\usepackage{algorithm}
\usepackage[pdftex,
pdftitle={MS BaCoN},
]{hyperref}
\hypersetup{
	colorlinks,
	allcolors=blue
}

\usepackage{cleveref} %

\usepackage{caption}
\usepackage{amsfonts}

\usepackage{lmodern}

\usepackage{thmtools}
\usepackage{thm-restate}

\providecommand{\corollaryname}{Corollary}
\providecommand{\definitionname}{Definition}

\providecommand{\propositionname}{Proposition}
\providecommand{\remarkname}{Remark}
\providecommand{\theoremname}{Theorem}
\theoremstyle{plain}
\newtheorem{thm}{\protect\theoremname}
\declaretheorem[name=Lemma,sibling=thm]{lem}
\theoremstyle{definition}
\newtheorem{defn}[thm]{\protect\definitionname}
\theoremstyle{plain}
\newtheorem{cor}[thm]{\protect\corollaryname}
\theoremstyle{plain}

\theoremstyle{remark}
\newtheorem*{rem*}{\protect\remarkname}

\global\long\def\R{\mathbb{R}}%

\global\long\def\boldVar#1{\mathbf{#1}}%
\global\long\def\mvar#1{\boldVar{#1}}%
\global\long\def\vvar#1{\vec{#1}}%

\global\long\def\defeq{\stackrel{\mathrm{{\scriptscriptstyle def}}}{=}}%
\global\long\def\otilde{\widetilde{O}}%

\global\long\def\gradient{\nabla}%
 
\global\long\def\grad{\gradient}%
 
\global\long\def\Hessian{\gradient^{2}}%
 
\global\long\def\hess{\Hessian}%

\global\long\def\indicVec#1{\onesVec_{#1}}%

\global\long\def\innerProduct#1#2{\left\langle#1 , #2 \right\rangle}%

\global\long\def\normFull#1{\left\Vert #1\right\Vert }%
\global\long\def\normm#1{\left\Vert #1\right\Vert_{\mm} }%
\global\long\def\normms#1{\Vert #1\Vert_{\mm} }%
\global\long\def\normmd#1{\left\Vert #1\right\Vert_{\mmd} }

\global\long\def\onesVec{\vvar 1}%

\global\long\def\ma{\mvar A}%

\global\long\def\md{\mvar D}%

\global\long\def\mg{\mvar G}%
 
\global\long\def\mh{\mvar H}%
\global\long\def\tmh{\tilde{\mvar H}}%
\global\long\def\mI{\mvar I}%
\global\long\def\tmi{\tilde{\mvar I}}%
\global\long\def\mm{\mvar M}%
\global\long\def\mmh{\mvar M^{1/2}}%
\global\long\def\mmd{\mvar M^\dagger}%
\global\long\def\mmdh{\mvar M^{\dagger/2}}%

\global\long\def\mU{\mvar U}%

\global\long\def\mha{\hat{\mvar A}}%

\global\long\def\mdiag{\mvar{diag}}%

\global\long\def\oracle{\mathcal{O}}%

\DeclareMathOperator*{\argmin}{arg\,min} %
\DeclareMathOperator*{\minimize}{minimize} %
\global\long\def\oracleBall{\oracle_{\textup{ball}}}%
\global\long\def\oracleLoc{\oracle_{\textup{local}}}%
\global\long\def\bx{\bar{x}}%
\global\long\def\ballmin{x_{\bx,r}}%
\global\long\def\eps{\epsilon}%
\global\long\def\half{\frac{1}{2}}%
\global\long\def\tztl{\tilde{z}_{t_{\lambda}}}%
\global\long\def\ztl{z_{t_{\lambda}}}%
\global\long\def\ytl{y_{t_{\lambda}}}%
\global\long\def\quadmin{x_{g,\mh}}%
\global\long\def\tx{\tilde{x}}%
\global\long\def\hx{\hat{x}}%
\global\long\def\xset{\mathcal{X}}%

\newcommand{\N}{\mathbb{N}}

\DeclarePairedDelimiter{\abs}{\lvert}{\rvert} 

\DeclarePairedDelimiter{\crl}{\{}{\}}
\DeclarePairedDelimiter{\prn}{(}{)}

\DeclarePairedDelimiter{\norm}{\|}{\|}
\DeclarePairedDelimiter{\tri}{\langle}{\rangle}

\DeclarePairedDelimiter{\ceil}{\lceil}{\rceil}
\DeclarePairedDelimiter{\floor}{\lfloor}{\rfloor}

\newcommand{\inner}[2]{\tri{#1,#2}}

\newcommand{\ball}[1]{\mathcal{B}_{#1}}
\newcommand{\ballset}{\ball{r}(\bx)}
\newcommand{\ballsett}[1]{\ball{r}(#1)}
\newcommand{\domain}{\ball{R}(0)}

\newcommand{\lse}{\textup{lse}}

\renewcommand{\P}{\mathbb{P}}
\newcommand{\indic}[1]{1\crl{#1}}
\newcommand{\ind}[1]{^{(#1)}}
\newcommand{\ones}{\boldsymbol{1}}
\newcommand{\omsf}{\oracle_{\textup{MS}}}
\newcommand{\lam}{\lambda}
\newcommand{\tg}{\tilde{g}}
\newcommand{\hg}{\hat{g}}
\usepackage{color}
\definecolor{burntorange}{rgb}{0.8, 0.33, 0.0}
\definecolor{darkgreen}{rgb}{0.0, 0.5, 0.0}
\definecolor{crayola}{rgb}{0.89, 0.043, 0.361}
\definecolor{burgundy}{rgb}{0.5, 0.0, 0.13}
\definecolor{purple}{rgb}{1, 0, 1}

\begin{document}

\pagenumbering{gobble}
\title{Acceleration with a Ball Optimization Oracle}
\author{Yair Carmon\thanks{Stanford University, \texttt{\{yairc,jmblpati,%
qjiang2,yujiajin,sidford,kjtian\}@stanford.edu}.}\\
\and 
Arun Jambulapati\footnotemark[1]\\
\and 
Qijia Jiang\footnotemark[1]\\
\and
Yujia Jin\footnotemark[1]\\
\and
Yin Tat Lee\thanks{University of Washington, \texttt{yintat@uw.edu}.}\\
\and
Aaron Sidford\footnotemark[1]\\
\and
Kevin Tian\footnotemark[1]\\
}
\date{}
\maketitle
\begin{abstract}
Consider an oracle which takes a point $x$ and returns the minimizer of a 
convex function $f$ in an $\ell_2$ ball of radius $r$ around $x$. It is 
straightforward to show that roughly $r^{-1}\log\frac{1}{\epsilon}$ calls to 
the oracle suffice to find an $\epsilon$-approximate minimizer of $f$ in an 
$\ell_2$ unit ball. Perhaps surprisingly, this is not optimal: we design an 
accelerated algorithm which attains an $\epsilon$-approximate minimizer 
with roughly $r^{-2/3} \log \frac{1}{\epsilon}$ oracle queries, and give a 
matching lower bound. Further, we implement ball optimization oracles for 
functions with locally stable Hessians using a variant of Newton's method. 
The resulting algorithm applies to a number of problems of practical and 
theoretical import, improving upon previous results for logistic and 
$\ell_\infty$ regression and achieving guarantees comparable to the  
state-of-the-art for $\ell_p$ regression.
\end{abstract}

\pagenumbering{arabic}
\setcounter{page}{1}

\section{Introduction}
\label{sec:intro}

We study unconstrained minimization of a smooth convex objective 
$f:\R^d\to\R$, which we access through a \emph{ball optimization oracle} 
$\oracleBall$, that when 
queried at any point $x$, returns the minimizer\footnote{In the 
introduction we discuss exact 
oracles for simplicity, but 
our results account for inexactness.} of $f$ restricted a ball
of 
radius $r$ around $x$, i.e.,
\[
\oracleBall(x) = \argmin_{x'~\text{s.t.}~\norm{x' - x} \leq r} f(x').
\]
Such oracles 
underlie trust region 
methods~\cite{conn2000trust} 
and, as we 
demonstrate via applications, encapsulate problems with local stability.
 Iterating $x_{k +1} \gets \oracleBall(x_k)$ minimizes $f$ in $\otilde(R/r)$ iterations 
(see~\Cref{sec:vanilla}), where $R$ is the initial distance to the 
minimizer, $x^*$, and $\otilde(\cdot)$ hides polylogarithmic 
factors in problem parameters, including the desired accuracy. 

Given the fundamental geometric nature of the ball optimization 
abstraction, the central question motivating our work is whether it is 
possible to improve upon this $\otilde(R/r)$ query complexity. It is natural 
to conjecture that the answer is negative: we require $R/r$ oracle calls to 
observe the entire line from $x_0$ to the optimum, and therefore finding a 
solution using less queries would require jumping into completely 
unobserved regions.
Nevertheless, we prove that the optimal query complexity scales as 
$(R/r)^{2/3}$. This result has positive 
implications for the complexity for several key regression tasks, for which we 
can efficiently implement the ball optimization oracles.

\subsection{Our contributions}

Here we overview the main contributions of our paper: accelerating
ball optimization oracles (with a matching lower bound), implementing 
them under Hessian stability, and 
applying the resulting techniques to regression problems.
\newcommand{\oracleMS}{\oracle_{\textup{MS}}}
\paragraph{Monteiro-Svaiter (MS) oracles via ball optimization.}
Our starting point is an acceleration framework due 
to~\citet{MonteiroS13a}. It relies on access to an oracle that when 
queried with $x,v\in\R^d$ and $A>0$, returns points $x_+,y\in\R^d$ and 
$\lambda>0$ such that
\begin{flalign}
y &= \frac{A}{A+a_\lambda} x + \frac{a_\lambda}{A+a_\lambda} 
v,~\mbox{and}
\label{eq:ms-y} \\
x_+ &\approx \argmin_{x'\in\R^d} \crl*{f(x') + 
\frac{1}{2\lambda}\norm{x'-y}^2},\label{eq:ms-prox}
\end{flalign}
where $a_\lambda = \frac{1}{2}(\lambda +  \sqrt{\lambda^2 +4 
A\lambda})$. Basic calculus shows that for any $z$, the radius-$r$ oracle 
response $\oracleBall(z)$ solves 
the 
proximal point problem~\eqref{eq:ms-prox} for $y=z$ and some  
$\lambda = \lambda_r(z)\ge 0$ which depends on $r$ and $z$. 
Therefore, to implement the MS oracle  
with a ball optimization oracle, we need to find $\lambda$ that solves the 
implicit 
equation $\lambda = \lambda_r(y(\lambda))$, with $y(\lambda)$ as 
in~\eqref{eq:ms-y}. We accomplish an approximate version of this via 
 binary search over $\lambda$, resulting in an accelerated scheme 
that makes $\otilde(1)$ queries to $\oracleBall(\cdot)$ per iteration (each 
iteration also requires a gradient evaluation).

The main challenge lies in proving that our MS oracle implementation 
guarantees rapid convergence. We do so by a careful analysis which relates 
convergence to the distance between the points $y$ and $x_+$ that the MS 
oracle outputs. Specifically, letting $\{y_k,x_{k+1}\}$ be the sequence of 
these points, we prove that
\begin{equation*}
\frac{f(x_{K})-f(x^*) }{f(x_0)-f(x^*)}\le \exp\crl*{-\Omega(K) \min_{k < K} 
\frac{\norm{x_{k+1}-y_k}^{2/3}}{R^{2/3}}}.
\end{equation*}
Since $\oracleBall$ guarantees $\norm{x_{k+1}-y_k} = r$ for all $k$ 
except possibly the last (if the final ball contains $x_*$), our result follows.

\paragraph{Matching lower bound.}
We give a distribution over functions with domain of size $R$ for which any 
algorithm interacting with a ball optimization oracle of radius $r$ requires 
$\Omega((R/r)^{2/3})$ queries to find an approximate solution with $O(r^{1/3})$ additive error. %
Our lower  bound in fact holds for an even more powerful $r$-local oracle, 
which reveals all values of $f$ in a ball of radius $r$ around the query 
point. We prove our 
lower bounds using well-established techniques and Nemirovski's 
function, a canonical hard instance in convex 
optimization~\cite{nemirovski1983problem,woodworth2017lower,
	carmon2019lower_i,diakonikolas2018lower,bubeck2019complexity}. 
	Here, our primary contribution is to show that
appropriately scaling this construction makes it hard even against 
$r$-local oracles with a fixed radius $r$, as opposed to the more standard 
notion of local oracles that reveal the instance only in an arbitrarily small 
neighborhood around the query. 

\paragraph{Implementation of a ball optimization oracle.}
Trust region methods \cite{conn2000trust}
 solve a sequence of subproblems of the form
\[
\minimize_{\delta\in\R^d~\text{s.t.}~\norm{\delta}\le r} \crl[\Big]{ 
\delta^\top g + 
\half 
\delta^\top \mh \delta}.
\]
When $g=\grad f(x)$ and $\mh=\hess f(x)$, the trust region subproblem 
minimizes a second-order Taylor expansion of $f$ around $x$, 
implementing an approximate ball optimization oracle. We show how to implement a ball optimization oracle for $f$ %
to high accuracy for functions satisfying a local Hessian stability 
property. Specifically, we use a notion of Hessian stability similar to that 
of~\citet{karimireddy2018global}, requiring $\frac{1}{c} \hess f(x) \preceq 
\hess f(y) 
\preceq c \hess f(x)$ for every $y$ in a ball of radius $r$ around $x$ for some $c>1$. We analyze Nesterov's accelerated gradient method in a  Euclidean norm weighted by the Hessian at $x$, which we can also view as accelerated Newton steps, and show that it 
implements the oracle in $\otilde(c)$ linear system solutions. Here 
acceleration improves upon the $c^2$ dependence of more naive methods. 
This improvement is not necessary for our applications where we take 
$c$ to be a constant, but we include it for completeness.

\paragraph{Applications.}
We apply our implementation and acceleration of ball optimization oracles  
to problems of the 
form $f(\ma x -b)$ for data matrix $\ma\in\R^{n\times d}$. For logistic 
regression, where 
\[
f(z)=\sum_{i\in[n]}\log(1+e^{-z_i}),
\]
the Hessian stability property~\cite{agarwal2017second} implies that our algorithm solves the problem with $\otilde(\norm{x_0-x^*}_{\ma^\top \ma}^{2/3})$ linear system solves of the form $\ma^\top \md \ma x = z$ for diagonal matrix $\md$. This improves upon the previous best linearly-convergent
condition-free algorithm due to~\citet{karimireddy2018global}, which requires 
$\otilde(\norm{x_0-x^*}_{\ma^\top \ma})$ system solves. Our  
improvement is precisely 
the power $2/3$ factor that comes from acceleration using the ball optimization
oracle.

For $\ell_\infty$ regression, we take $f$ to be the log-sum-exp (softmax) 
function 
and establish that it too has a stable Hessian. By appropriately scaling softmax to  approximate $\ell_\infty$ to $\epsilon$ additive error and taking $r=\epsilon$, we obtain an algorithm that solves $\ell_\infty$ to additive error $\epsilon$ in $\otilde(\norm{x_0-x^*}_{\ma^\top 
\ma}^{2/3}\epsilon^{-2/3})$ linear system solves of the same form as 
above. This improves upon the algorithm of~\citet{bullins2019higher} 
which 
requires $\otilde(\norm{x_0-x^*}_{\ma^\top 
\ma}^{4/5}\epsilon^{-4/5})$ linear system solves.

Finally, we leverage our implementation of a ball optimization oracle to 
obtain high accuracy solutions to $\ell_p$ norm regression, where 
$f(z)=\sum_{i\in[n]} |z_i|^p$. Here, we use our accelerated 
ball-constrained Newton algorithm to minimize a sequence of proximal 
problems with a geometrically shrinking quadratic regularization term. 
The result is an algorithm that solves  $\otilde(\mathsf{poly}(p) n^{1/3})$ 
linear systems.  For $p = \omega(1)$, this matches the state-of-the-art 
$n$ dependence~\cite{AS20}
but obtains a worse dependence on $p$. Nevertheless, our approach seems simpler than prior work and leaves room for further refinements which we believe will result in stronger guarantees.

\subsection{Related work}\label{sec:related}
Our developments are rooted in three lines of work, which we now briefly 
survey.

\paragraph{Monteiro-Svaiter framework instantiations.}
\citet{MonteiroS13a} propose a new acceleration framework, which they 
specialize to recover the classic fast gradient method~\cite{Nesterov83} 
and obtain an optimal accelerated second-order method for convex 
problems with Lipschitz Hessian. Subsequent
work~\cite{GasnikovDGVSU0W19} extends this to functions with $p$th-order Lipschitz derivatives and a $p$th-order oracle. Generalizing further, \citet{bubeck2019complexity} implement the MS 
oracle via a ``$\Phi$ prox'' oracle that given query $x$ returns roughly 
$\argmin_{x'}\{ f(x) + \Phi(\norm{x'-x})\}$, for continuously differentiable $\Phi$, and prove an error bound scaling 
with the iterate number $k$ as $\phi(R/k^{3/2})R^2/k^2$, where 
$\phi(t) = \Phi'(t)/t$. 
Using $\mathsf{poly}(d)$ parallel queries to a 
subgradient oracle for non-smooth $f$, they show how to implement the 
$\Phi$ prox oracle for $\Phi(t)\propto (t/r)^p$ with arbitrarily large $p$, 
where $r=\epsilon/\sqrt{d}$. 
Our notion of a ball optimization corresponds to taking $p=\infty$, i.e., 
letting $\Phi$ be the indicator of $[0,r]$. However, since such $\Phi$ is not 
continuous, our result does not follow directly 
from~\cite{bubeck2019complexity}. Thus, our approach clarifies 
the limiting behavior of MS acceleration of infinitely smooth functions.

\paragraph{Trust region methods.}
The idea of approximately minimizing the objective in a ``trust 
region'' around the current iterate plays a central role in nonlinear 
optimization and 
machine learning~\citep[see, e.g.,][]{conn2000trust,LinWK08,SchmidtKS11}. 
Typically, the approximation takes the form of a second-order Taylor 
expansion, where regularity of the Hessian is key for guaranteeing the 
approximation quality. Of particular relevance to us is the work of 
\citet{karimireddy2018global}, which define a notion of Hessian stability 
under which a 
trust region method converges linearly with only logarithmic 
dependence on problem conditioning. 
 We observe that this stability 
condition in fact renders the second-order trust region approximation 
highly effective, so that a few iterations suffice in order to implement an 
``ideal'' ball optimization oracle, thus enabling accelerated condition-free 
convergence.

\paragraph{Efficient $\ell_p$  regression algorithms.}
There has been rapid recent progress in linearly convergent algorithms for minimizing the 
$p$-norm of the regression residual $\ma x-b$ or alternatively for finding 
a minimum $p$-norm $x$ satisfying the linear constraints $\ma x = b$. 
\citet{BubeckCLL18} give faster algorithms for all $p \in (1, 2) \cup (2, 
\infty)$, discovering and overcoming a limitation of classic interior point 
methods. 
Their algorithm is based on considering a ``smoother'' objective which 
behaves as a quadratic within a region, and as the original $p$th-order 
objective outside. \citet{AdilKPS19} improve on this result with an 
algorithm with iteration complexity bounded by $n^{1/3}$ (for 
regression in $n$ constraints) for all $p$ bounded away from $1$ and 
$\infty$, improving upon the $n^{1/2}$ limit behavior of 
\citet{BubeckCLL18}. \citet{AS20} provide an alternative method which 
achieves $n^{1/3}$ iterations with a linear dependence on $p$, improving 
on the $O(p^{O(p)})$ dependence found in \citet{AdilKPS19}. 
For $p=\infty$, \citet{bullins2019higher} develop a method based on 
fourth-order MS 
acceleration for $\eps$-approximately minimizing the smooth softmax 
approximation to the $\ell_\infty$ objective, with iteration complexity 
$\eps^{-4/5}$. 
We believe that our approach brings us closer to a unified perspective on high-order smoothness and acceleration for regression problems. 

\subsection{Paper organization}
In \Cref{sec:accelerated_ball_framework}, we implement the MS oracle 
using a ball optimization oracle and prove its $\otilde((R/r)^{2/3})$ 
convergence guarantee. In \Cref{sec:stable_Hessian}, we show how to use 
Hessian stability to efficiently implement a ball optimization oracle, and 
also show that quasi-self-concordance implies Hessian stability. In 
\Cref{sec:applications} we apply our developments to the aforementioned 
regression tasks. Finally, in \Cref{sec:lb} we give a lower bound implying our 
oracle complexity is optimal (up to logarithmic 
terms).

\paragraph{Notation.}
Let $\mm$ be a positive semidefinite matrix, and let $\mmd$ be its 
pseudoinverse. We perform our analysis in the Euclidean seminorm 
$\norm{x}_{\mm} \defeq \sqrt{x^\top \mm x}$; we will choose a specific 
$\mm$ when discussing applications. We denote the $\norm{\cdot}_{\mm}$ 
ball of radius $r$ around $\bx$ by
\[
\ballset\defeq\left\{ x\in\R^d\mid\normm{x-\bx}\le r\right\} .
\]
We recall standard definitions of smoothness and strong-convexity in a quadratic norm: differentiable $f:\R^d \rightarrow \R$ is $L$-smooth in $\normm{\cdot}$ if its gradient is $L$-Lipschitz in $\normm{\cdot}$, and twice-differentiable $f$ is $L$-smooth and $\mu$-strongly convex in $\normm{\cdot}$ if $\mu \mm \preceq \nabla^2 f(x) \preceq L\mm$ for all $x \in \R^d$.

\section{Monteiro-Svaiter Acceleration with a Ball Optimization Oracle}
\label{sec:accelerated_ball_framework}

In this section, we give an accelerated algorithm for optimization with the 
following oracle.
\begin{defn}[Ball optimization oracle]
	\label{def:ball_opt}  We call $\oracleBall$ a \emph{$(\delta,r)$-ball
	optimization oracle} for $f:\R^{d}\rightarrow\R$ if
	for any $\bar{x}\in\R^{d}$, it 
	outputs $y=\oracleBall(\bx)\in\ballset$ such that 
	$\norm{y-x_{\bx,r}}_\mm\le \delta$  for some 
	$x_{\bx,r}\in\argmin_{x\in\ballset}f(x)$.
\end{defn}

Our algorithm utilizes the acceleration framework of 
\citet{MonteiroS13a} (see also  \citep{GasnikovDGVSU0W19, 
bubeck2019complexity}). It relies on the 
following oracle.
\begin{defn}[MS oracle]\label{def:ms-oracle}
We call $\omsf$ a $\sigma$-MS oracle for differentiable $f:\R^{d}\rightarrow\R$ if given inputs 
$(A, x, v) \in \R_{\ge 0} \times \R^d \times \R^d$, $\omsf$ outputs 
$(\lambda, a_\lambda, y_{t_\lambda}, z) \in \R_{\ge 0}\times  \R_{\ge 0}
\times \R^{d}\times \R^{d}$ 
such that
\begin{flalign*}
&a_{\lambda} = \frac{\lambda + \sqrt{\lambda^2 + 4\lambda A}}{2},~
t_{\lambda} = \frac{A}{A+a_\lambda},~
 y_{t_\lambda} = t_\lambda \cdot x + (1-t_\lambda)\cdot v,
\end{flalign*}
and we have the guarantee 
\begin{equation}\label{eq:ms_requirement}\normm{z - (\ytl - \lambda 
\mmd \nabla f(z))} \le \sigma\normm{z - \ytl}.\end{equation}
\end{defn}
We now state the acceleration framework and the main bound we use to 
analyze its convergence.

\begin{algorithm}
\begin{algorithmic}[1]
\caption{Monteiro-Svaiter acceleration}
\label{alg:msaccel_framework}
\STATE{\textbf{Input: }Strictly convex and differentiable function 
$f:\R^{d}\rightarrow\R$. }%
\STATE{\textbf{Input: }Symmetric $\mm\succeq 0$ 
with $\nabla f(x) \in \text{Im}(\mm)$ for all $x \in \R^d$.}
\STATE{\textbf{Input: }Initialization $A_{0}\ge0$ and 
$x_{0}=v_{0}\in\R^{d}$.}
\STATE{\textbf{Input: }Monteiro-Svaiter oracle $\omsf$ with parameter 
$\sigma\in[0,1)$.}
\FOR{$k=0, 1,2,\ldots$}
\STATE{$(\lambda_{k + 1}, a_{k + 1}, y_k, x_{k + 1}) \gets \omsf(A_k, x_k, v_k)$}
\STATE{$v_{k+1} \gets v_{k}-a_{k+1}\mmd\grad f(x_{k+1})$.}%
\STATE{$A_{k + 1} \gets A_k + a_{k + 1}$}
\ENDFOR
\end{algorithmic}
\end{algorithm}
\begin{restatable}{prop}{convergencerateaccel}
	\label{prop:ball_constrained_error_bound} 
	Let $f$ be strictly convex and differentiable, with $\normm{x_0 - x^*} 
	\le R$ and $f(x_0) - f(x^*) \le \eps_0$. Set $A_0 = R^2/(2\epsilon_0)$ 
	and suppose that  for some $r>0$ the iterates of 
	\Cref{alg:msaccel_framework} satisfy 
	\[\normm{x_{k+1}-y_{k}}\geq r~\mbox{for all}~k\ge0.\]
	Then, the iterates also satisfy
	\[
	f(x_k) - f(x^*)\leq2\eps_0\exp\left(-\left(\frac{r(1-\sigma)}{R}\right)^{2/3}(k-1)\right)\cdot
	\]
\end{restatable}
\Cref{prop:ball_constrained_error_bound} is one of our main technical 
results. We  obtain it by applying a reverse H\"older's inequality on a variant of the 
performance guarantees of~\citet{MonteiroS13a}; 
we defer the proof to Appendix~\ref{sec:framework}. 

Clearly, \Cref{prop:ball_constrained_error_bound} implies that the progress 
of Algorithm \ref{alg:msaccel_framework} is related to the amount of 
movement of the iterates, i.e., the quantities $\{\normm{x_{k+1}-y_{k}}\}$.
We now show that by using a ball optimization oracle over radius $r$, we 
are able to guarantee movement by roughly $r$, which implies rapid 
convergence. We rely on the following characterization, whose proof we 
defer to \Cref{sec:lslipschitz}.
\begin{restatable}{lem}{balloptchar}
	\label{lem:opt_characterization} Let $f: \R^d \rightarrow \R$ be 
	continuously differentiable and strictly convex. For all $y\in\R^{d}$, 
	\[z=\argmin_{z' \in \ballsett{y}} f(z')\]
	is either the global minimizer of $f$, or $\normm{z-y}=r$
	and $\grad f(z) =  -\frac{\normmd{\grad f(z)}}{r} \mm (z-y)$.
\end{restatable}
Lemma~\ref{lem:opt_characterization} implies that a $(0,r)$ ball 
optimization oracle either gives 
$z=\oracleBall(y)$ globally minimizing $f$, or 
\begin{flalign}
&\normm{z-y}=r~\mbox{and}~\normm{z-\left(y- \lambda \mmd \nabla 
	f(z)\right)}=0, \quad \text{for } \lambda = \frac{r}{\normmd{\nabla f(z)}}. 
\label{eq:lamdef}
\end{flalign}
This is precisely the type of bound compatible with both 
Proposition~\ref{prop:ball_constrained_error_bound} and  
requirement~\eqref{eq:ms_requirement} of $\omsf$. The remaining  
difficulty lies in 
that $\lambda$ also defines the point $y=y_{t_\lambda}$. 
Therefore, 
to implement an MS oracle using a ball optimization oracle we perform 
binary search over $\lambda$, with the goal of solving 
\[
 g(\lambda)\defeq\lambda\norm{\nabla f(\ztl)}_{\mm^\dagger}=r,
 ~\mbox{where}~
 \ztl \defeq \min_{z\in\ballsett{\ytl}} f(z),
 \]
 and $t_\lambda, y_{t_\lambda}$ are as in \Cref{def:ms-oracle}. 
 
 \Cref{alg:line-search-oracle} describes our binary search implementation. 
 The algorithm 
 accepts the MS oracle input $(A,x,v)$ as well as a bound $D$ on the 
 distance of $x$ and $v$ from the optimum, and desired global solution 
 accuracy $\epsilon$, and outputs either a (globally) $\epsilon$-approximate minimizer
 or a tuple $(\lambda,a_\lambda,\ytl, \tztl)$ satisfying 
 both~\eqref{eq:ms_requirement} (with $\sigma=\frac{1}{2}$) and a lower 
 bound on the distance between $\tztl$ and $\ytl$. To bound the 
 complexity of our procedure we leverage $L$-smoothness of $f$ (i.e., 
 $L$-Lipschitz continuity of $\grad f$), which allows us to bound the 
 Lipschitz constant of $g(\lambda)$ defined above. The analysis of the 
 algorithm is somewhat intricate because of the need to account for 
 inexactness in the ball optimization oracle. It results in the following 
 performance guarantee, whose proof we defer to \Cref{sec:lslipschitz}.
 
\begin{restatable}[Guarantees of Algorithm 
	\ref{alg:line-search-oracle}]{prop}{msoracleguarantee}
	\label{prop:line_search} 
	Let $L,D,\delta,r>0$ and $\oracleBall$ satisfy the requirements in Lines 
	1--3 of 
	Algorithm~\ref{alg:line-search-oracle}, and $\eps < 
	2LD^{2}$. Then, Algorithm~\ref{alg:line-search-oracle} either returns 
	$\tztl$ with $f(\tztl) - f(x^*) < \eps$, or implements a $\half$-MS oracle 
	with the additional guarantee
	\[\normm{\tztl - \ytl} \ge \frac{11r}{12}.\]
	Moreover, the number of calls to $\oracleBall$
	is bounded by
	\[
	O\left(\log\left(\frac{LD^{2}}{\eps}\right)\right).
	\]
\end{restatable}

\begin{algorithm}[ht!]
\begin{algorithmic}[1]
\caption{Monteiro-Svaiter oracle 
implementation}\label{alg:line-search-oracle}

\STATE{\textbf{Input: }Function $f:\R^{d}\rightarrow\R$ that is strictly 
convex, $L$-smooth in $\normm{\cdot}$.}

\STATE{\textbf{Input: }$A\in\R_{\ge0}$ and $x,v\in\R^d$ satisfying 
 $\normm{x-x^{*}}$ and $\normm{v-x^{*}}\le D$ where $x^*=\argmin_x 
 f(x)$.}

\STATE{\textbf{Input: }A $\left(\delta,r\right)$-ball optimization 
	oracle $\oracleBall$, where $\delta=\frac{r}{12(1+Lu)}$ and $u=\frac{2(D 
	+ r)r}{\eps}$}

\STATE{Set $\lambda\gets u$ and $\ell\gets\frac{r}{2LD}$}

\STATE{$\tztl\gets\oracleBall(\ytl)$}

\IF{$u\normmd{\nabla f(\tztl)}\le r+uL\delta$}

\RETURN $(\lambda, a_\lam, \ytl, \tztl)$

\ELSE

\WHILE{$\left|\lambda\normmd{\nabla f(\tztl)}-r\right|>\frac{r}{6}$}

\STATE{$\lambda\gets\frac{\ell+u}{2}$}

\STATE{$\tztl\gets\oracleBall(\ytl)$}

\IF{$\lambda\normmd{\nabla f(\tztl)}\ge r$}

\STATE{$u\gets\lambda$}%

\ELSE

\STATE{$\ell\gets\lambda$}%

\ENDIF

\ENDWHILE

\RETURN $(\lambda, a_\lam, \ytl, \tztl)$
\ENDIF
\end{algorithmic}
\end{algorithm}
Finally, we state our main acceleration result, whose proof we defer to 
\Cref{sec:lslipschitz}.
\begin{restatable}[Acceleration with a ball optimization 
oracle]{thm}{mainclaimaccel}
	\label{thm:mainclaim-accel}  
	Let $\oracleBall$ be an
	$(\frac{r}{12+126LRr/\eps},r)$-ball
	optimization oracle for strictly convex and $L$-smooth
	$f:\R^{d}\rightarrow\R$ with minimizer $x^{*}$, and initial point $x_0$ 
	satisfying 
		\[\normm{x_{0}-x^{*}}\le R~\mbox{and}~ f(x_0) - f(x^*) \le \eps_0.\]
	Then, 
	Algorithm~\ref{alg:msaccel_framework} using 
	Algorithm~\ref{alg:line-search-oracle} as a Monteiro-Svaiter oracle with $D=\sqrt{18}R$ 
	produces an iterate $x_k$ with $f(x_k)-f(x^{*})\le\eps$, in
	\[
	O\left({{\left(\frac{R}{r}\right)^{2/3}
			\log\bigg(\frac{\eps_0}{\eps}\bigg)}}
		\log\bigg(\frac{LR^{2}}{\eps}\bigg)\right)
	\]
	calls to $\oracleBall$.
\end{restatable}

\section{Ball Optimization Oracle for Hessian Stable Functions}

\label{sec:stable_Hessian}

In this section, we give an implementation of a ball optimization oracle 
$\oracleBall$ for functions satisfying the following notion of Hessian 
stability, which is a slightly stronger version of the condition in~\citet{karimireddy2018global}.\footnote{
A variant of the algorithm we develop also works under the 
weaker stability condition. We state the stronger condition as it is simpler, and holds for all our applications.
}
\begin{defn}[Hessian stability]
	A twice-differentiable function $f:\R^d\to\R$ is 
	\emph{$(r,c)$-Hessian
		stable }for $r,c\geq0$ with respect to norm $\norm{\cdot}$
	if for all $x,y\in\R^{d}$ with $\norm{x-y}\leq r$ we have 
	$c^{-1}\hess f(y)\preceq\hess f(x)\preceq c\hess f(y)$.%
\end{defn}

We give a method that 
implements a $(\delta,r)$-ball oracle (as in \Cref{def:ball_opt}) for $(r, 
c)$-stable functions in $\normm{\cdot}$, requiring 
$\otilde(c)$ linear system solutions.  Our method's complexity has a (mild) 
polylogarithmic dependence on the \emph{condition number} of $f$ in 
$\normm{\cdot}$. The main result of this section is  
Theorem~\ref{thm:ballopt}, which guarantees the correctness and complexity 
our ball optimization oracle implementation. 
We prove it in two parts:  first, we provide a convergence guarantee for 
trust region subproblems, and then use it as a primitive in Algorithm~\ref{alg:nesterov_h}, an accelerated ball-constrained Newton's method. Finally, we describe a sufficient condition for Hessian stability to hold.

\subsection{Trust region subproblems}
\label{sec:quad_constrained_quad}

We describe a procedure for solving the convex trust region problem
\begin{equation*}
\min_{x\in\ballsett{\bx}}Q(x)\defeq-g^{\top}x+\half x^{\top}\mh x.\label{eq:qcqo}
\end{equation*}
While trust region problems of this form are well-studied 
\cite{conn2000trust, Hager01}, we could not find a concrete bound on the 
number of linear system solutions required to solve them approximately. In 
Appendix~\ref{app:trust-region} we describe the procedure
$\textsc{SolveTR}(\bx, r, g, \mh, \mm, \Delta)$ 
(Algorithm~\ref{alg:tr-method}) that uses a well-known binary search 
strategy to solve the trust region problem to accuracy $\Delta$. The 
procedure enjoys the following convergence guarantee.
\begin{restatable}{prop}{ghsolve}
\label{prop:ghsolve}
Let $\mm$ and $\mh$ share a kernel, $\mu \mm \preceq \mh$ for $\mu > 0$, and let $\Delta>0$. The 
procedure $\textsc{SolveTR}(\bx, r, g, \mh, \mm, \Delta)$ solves 
\[
O\left(\log\left(\frac{\normmd{\mh\bx-g}^{2}}{r\mu^{2}\Delta}\right)\right)
\]
linear systems in matrices of the form $\mh+\lambda\mm$ for 
$\lambda\ge0$, and returns $\tx\in\ballset$ with $\normm{\tx-\quadmin}\le\Delta$,
where 
\[
\quadmin\in\argmin_{x \in \ballset}-g^{\top}x+\half x^{\top}\mh x.
\]
\end{restatable}

\subsection{Ball-constrained Newton's method}

Theorem \ref{thm:ballopt} follows from an analysis of 
Algorithm~\ref{alg:nesterov_h}, which is essentially 
Nesterov's accelerated gradient method in the Euclidean seminorm 
$\norm{\cdot}_\mh$ with $\mh = \nabla^2 f(\bx)$,
or equivalently a sequence of constrained Newton steps using
the Hessian of the center point $\bx$. Other works \cite{DevolderGN14, 
CohenDO18} consider variants of Nesterov's accelerated method in arbitrary 
norms and under various noise assumptions, but do not give convergence 
guarantees compatible with the type of error incurred by our trust region 
subproblem solver. We state the convergence guarantee below, and defer its 
proof to Appendix~\ref{app:newton-appendix} for completeness; it is a 
simple adaptation of the standard acceleration analysis under inexact 
subproblem solves.

\begin{restatable}{thm}{ballopt}\label{thm:ballopt}
	Let $f$ be $L$-smooth, $\mu$-strongly convex, and $(r, c)$-Hessian 
	stable in the seminorm $\normm{\cdot}$. Then, 
	Algorithm~\ref{alg:nesterov_h} implements a $(\delta, r)$-ball 
	optimization oracle for query point $\bx$ with $\normm{\bx-x^*}\le 
	D$ for $x^*$ the minimizer of $f$, and requires 
	\[
	O\left(c\log^2\left(\frac{\kappa (D+r)c}{\delta}\right)\right)
	\]
	linear system solves in matrices of the form $\mh+\lambda\mm$ for nonnegative $\lam$, where $\kappa = L/\mu$.
\end{restatable}

\begin{algorithm}
\caption{Accelerated Newton's method}\label{alg:nesterov_h}
\begin{algorithmic}[1]
\STATE{\textbf{Input:} Radius $r$ and accuracy $\delta$ such that $r \ge \delta > 0$.}
\STATE{\textbf{Input:} Function $f:\R^d\rightarrow\R$ that is $L$-smooth, $\mu$-strongly convex, and $(r, c)$-Hessian stable in $\normm{\cdot}$ with minimizer $x^*$.}
\STATE{\textbf{Input:} Center point $\bx\in\R^d$ satisfying $\normms{\bx - x^*} \le D$.}
\STATE{$\mh \gets \nabla^2 f(\bx)$}
\STATE {$\alpha\gets c^{-1}$, $\Delta \gets \frac{\mu\delta^2}{4Lc(5r+D)}$, $x_0 \gets \bx$, $z_0\gets \bx$}
\FOR {$k=0, 1,2,\ldots$}
\STATE{$y_{k}\gets\frac{1}{1+\alpha}x_{k}+\frac{\alpha}{1+\alpha}z_{k}$} 
\STATE $g_k \gets \nabla f(y_k) - \mh(\alpha y_k + (1 - \alpha)z_k)$ 
\STATE{$z_{k+1}\gets\textsc{SolveTR}(\bx, r, g_k, \mh, \mm, \Delta)$} \label{line:AGD-solve-sub}
\STATE{$x_{k+1}\gets\alpha z_{k+1}+(1-\alpha)x_{k}$}
\ENDFOR
\end{algorithmic}
\end{algorithm}

\subsection{Quasi-self-concordance implies Hessian stability}

We state a sufficient condition for Hessian
stability below.
We use this result in 
Section~\ref{sec:applications} to establish Hessian stability in 
several structured problems. 
\begin{defn}[Quasi-self-concordance]
 We say that thrice-differentiable $f:\R^{d}\rightarrow\R$ is 
 $M$-quasi-self-concordant (QSC)
with respect to some norm $\norm{\cdot}$, for $M\geq0$, if for all
$u,h,x\in\R^{d}$,
\[
\left|\nabla^{3}f(x)[u,u,h]\right|\le M\norm h\norm u_{\nabla^{2}f(x)}^{2},
\]
i.e., the restriction
of the third-derivative tensor of $f$ to any direction is bounded by a multiple
of its Hessian norm.
\end{defn}

\begin{restatable}{lem}{quasisc}
\label{lem:quasi-to-stable}If thrice-differentiable $f:\R^{d}\rightarrow\R$
is $M$-quasi-self-concordant with respect to norm $\norm{\cdot}$,
then it is $(r,\exp(Mr))$-Hessian stable with respect to $\norm{\cdot}$.
\end{restatable}
\noindent
For completeness, we provide a proof in Appendix~\ref{app:qsc}. 

\section{Applications}

\label{sec:applications}

Algorithm~\ref{alg:ms-bacon} puts together the ingredients of the previous 
section to give a complete second-order method for minimizing QSC 
functions. 
In this section, we apply it to functions of the form $f(x) = 
g(\ma x)$ for a matrix $\ma \in \R^{n\times d}$ and function $g: \R^n 
\rightarrow \R$. The logistic loss function, the softmax 
approximation of the $\ell_\infty$ regression objective, and variations of 
$\ell_p$ regression objectives, all have this form. The following complexity 
guarantee for Algorithm~\ref{alg:ms-bacon} follows directly from our 
previous developments and we defer its proof to 
Appendix~\ref{app:application}.

\begin{algorithm}
	\caption{Monteiro-Svaiter Accelerated BAll COnstrained Newton's method 
	($\texttt{MS-BACON}$)}\label{alg:ms-bacon}
	\begin{algorithmic}[1]
		\STATE{\textbf{Input:} Function $f:\R^d \to \R$, desired accuracy 
		$\epsilon$, initial point $x_0$, initial suboptimality $\eps_0$.}
		\STATE{\textbf{Input:} Domain bound $R$, quasi-self-concordance 
		 $M$, smoothness  $L$, norm $\norm{\cdot}_{\mm}$.}
		\STATE{Define $\tilde{f}(x) = f(x) + \frac{\eps}{55R^2}\normm{x - 
		x_0}^2$}
		\STATE{Using Algorithm~\ref{alg:nesterov_h}, implement 
		$\oracleBall$, a $(\delta , \frac{1}{M})$-ball optimization oracle for 
		$\tilde{f}$, where $\delta = \Theta(\frac{\epsilon}{LR})$}
		\STATE{Using Algorithm~\ref{alg:line-search-oracle} and 
		$\oracleBall$, implement $\oracleMS$, a $\frac{1}{2}$-MS oracle for 
		$\tilde{f}$}
		\STATE{Using $O((RM)^{2/3}\log\frac{\eps_0}{\eps})$ iterations of 
		Algorithm~\ref{alg:msaccel_framework} with $\oracleMS$ 
		and initial point $x_0$ compute $x_{\mathrm{out}}$, an 
		$\epsilon/2$-accurate minimizer of $\tilde{f}$}
		\RETURN{$x_{\mathrm{out}}$}
	\end{algorithmic}
\end{algorithm}

\begin{restatable}{cor}{citethis}
\label{corr:citethis}
Let $f(x) = g(\ma x)$, for  $g: \R^n \rightarrow \R$ that is $L$-smooth, 
$M$-QSC  in the 
$\ell_2$ norm, and $\ma \in \R^{n \times d}$. Let $x^*$ be a minimizer of 
$f$, and suppose that
$\normm{x_0 - x^*} \le R$ and $f(x_0) - f(x^*) \le \eps_0$ 
for some $x_0 \in \R^d$, where $\mm \defeq \ma^\top\ma$. 
Then, Algorithm~\ref{alg:ms-bacon} 
yields an $\eps$-approximate minimizer to $f$ in
\[O\left(\left(RM\right)^{2/3}\log\left(\frac{\eps_0}{\eps}\right)\log^3\left(\frac{LR^2}{\eps}(1 + RM)\right)\right)\]
linear system solves in matrices of the form
\[\ma^\top\left(\nabla^2 g(\ma x) + \lam \mI\right)\ma
~~\text{for $\lambda \ge 0$ and $x \in \R^d$.}\]

\end{restatable}

\subsection{Logistic regression}
Consider \emph{logistic regression} with a
data matrix $\ma\in \mathbb{R}^{n\times d}$ with $n$ data points of dimension $d$, and corresponding labels $b\in 
\{-1,1\}^{n}$. %
The objective is%
\begin{equation}
\label{eqn:logistic}
f(x)= \sum_{i\in[n]}\log(1+\exp(-b_{i}\langle a_{i},x\rangle)) = g(\ma x),
\end{equation}
where $g(y) = \sum_{i \in [n]} \log(1 + \exp(-b_i y_i))$. It is known \citep{bach2010} that $g$ is 1-QSC and 1-smooth in $\ell_2$, with a diagonal Hessian. Thus, we have the following convergence guarantee from Corollary~\ref{corr:citethis}.
\begin{cor}
For the logistic regression objective in \eqref{eqn:logistic}, given $x_0$ 
with initial function error $\eps_0$ and distance $R$ away from a 
minimizer in $\norm{\cdot}_{\ma^\top\ma}$, 
Algorithm~\ref{alg:ms-bacon} obtains an $\eps$-approximate minimizer 
using
\[O\left(R^{2/3}\log\left(\frac{\eps_0}{\eps}\right)\log^3\left(\frac{R^2}{\eps}(1 + R)\right)\right)\]
linear system solves in matrices $\ma^\top\md\ma$ for diagonal $\md$.
\end{cor}
Compared to \citet{karimireddy2018global}, which gives a trust region 
Newton method using $\otilde(R)$ linear system solves, we obtain an 
improved dependence on the domain size $R$.

\subsection{$\ell_\infty$ regression}

Consider \emph{$\ell_\infty$ regression} in the matrix $\ma \in \R^{n\times d}$ and vector $b \in\R^n$, which asks to minimize the objective
\begin{equation}\label{eq:linfdef}f(x) = \norm{\ma x-b}_{\infty} = g(\ma x),\end{equation}
where $g(y) = \norm{y - b}_{\infty}$. Without loss of generality (by concatenating $\ma$, $b$ with $-\ma$, $-b$), we may replace the $\norm{\cdot}_\infty$ in the objective with a maximum. It is well-known that $g(y)$ is approximated within additive $\eps/2$ by $\lse_t(y - b)$ for $t = \eps/2\log n$ (see Lemma~\ref{lem:lseapprox} for a proof), where 
\[\lse(x)\defeq\log\left(\sum_{i\in[n]}\exp(x_{i})\right),\; \lse_t(x) \defeq t\lse(x/t).\]

Our improvement stems from the fact that $\lse_t$ is QSC which to the best of our knowledge was previously unknown. The proof consists of careful manipulation of the third derivative tensor of $\lse_t$ and is deferred to Appendix~\ref{app:application}.
\begin{lem}
\label{lem:softmax_stable}
$\lse_t$ is $1/t$-smooth and $2/t$-QSC in $\ell_{\infty}$.
\end{lem}
Lemma~\ref{lem:softmax_stable} immediately implies that $\lse_t$ is 
$n/t$-smooth and $2/t$-QSC in $\ell_2$, as for all $y\in\R^n$, 
$\norm{y}_\infty \leq \norm{y}_2 \le \sqrt{n}\norm{y}_\infty$, which clearly still holds under 
linear shifts by $b$. We thus obtain the following by applying 
Corollary~\ref{corr:citethis} to the $\lse_{\eps/2}$ objective, and solving to 
$\eps/2$ additive accuracy.
\begin{cor}
\label{cor:softmax}
For the $\ell_\infty$ regression objective in \eqref{eq:linfdef}, given $x_0$ 
with initial function error $\eps_0$ and $R$ away from a minimizer in 
$\norm{\cdot}_{\ma^\top\ma}$, Algorithm~\ref{alg:ms-bacon} obtains an 
$\eps$-approximate minimizer using
\[O\left(\left(\frac{R\log n}{\eps}\right)^{2/3}\log\left(\frac{\eps_0}{\eps}\right)\log^3\left(\frac{nR}{\eps}\right)\right)\]
linear system solves in matrices $\mha^{\top}(\mh+\lambda\mI)\mha$, where $\mh$ is a scaled Hessian of the $\lse$ function, $\lam\ge0$, and $\mha$ is the concatenation of $\ma$ and $-\ma$.
\end{cor}

Compared to \citet{bullins2019higher}, which obtains an 
$\epsilon$-approximate
solution to \eqref{eq:linfdef} in $\otilde((R/\epsilon)^{4/5})$ linear system solves using high-order acceleration, we obtain an improved dependence on $R/\epsilon$.

\subsection{$\ell_{p}$ regression}
\label{subsection:lp_reg}

Consider $\ell_p$ regression in the matrix $\ma \in \R^{n \times d}$ and vector $b \in \R^n$, which asks to minimize
\begin{equation}\label{eq:lpdef}
    f(x) = \norm{\ma x - b}_p^p = g(\ma x).
\end{equation}
for some fixed $p > 3$,\footnote{We assume $p > 3$ for ease of 
presentation; for $p \leq 4$ our runtime is superseded by, e.g., the 
algorithm 
of \cite{AdilPS19}.} where $g(x) = \sum_i \lvert x_i - b_i\rvert^p$. We refer 
to the optimal value of \eqref{eq:lpdef} by $f^*$, and its minimizer by 
$x^*$; we will solve \eqref{eq:lpdef} to $1+\delta$ multiplicative accuracy. 
By taking $p$th roots and solving to an appropriate lower accuracy level, 
this also recovers more standard formulations of minimizing $\norm{\ma x 
- b}_p$.

Prior work on this problem shows \eqref{eq:lpdef} can be 
minimized using fewer than the $O(n^{1/2})$ linear system solves that an 
interior point method would require: the state of the art algorithms of 
\citet{AS20, AdilKPS19} minimize $f$ to $1+\delta$ multiplicative 
accuracy by solving $\otilde\left(\min\left(pn^{1/3},p^{O(p)} 
n^{\frac{p-2}{3p-2}}\right) \log(1/\delta)\right)$ linear systems in 
 $\ma^\top \md \ma$ where $\md$ is a positive semidefinite
 diagonal matrix. In 
this section we provide an algorithm to minimize $g$ in $\otilde(p^{14/3} 
n^{1/3} \log^4(n/\delta))$ such systems. While our techniques do not 
improve on the state of the art, we believe our proof and algorithm are 
simpler than the previous work and of independent interest.

Algorithm~\ref{alg:lp_reg} summarizes our approach. It consists iteratively 
applying Algorithm~\ref{alg:ms-bacon} to the 
objective~\eqref{eq:lpdef} with exponentially shrinking target additive 
error. We initialize the algorithm at $x_0 = \argmin_x \norm{\ma x - b}_2$. 
Using the fact that $\norm{y}_2 \le n^{(p - 2)/2p}\norm{y}_p$ for all $y$ 
and $p$, the 
initialization satisfies
\begin{equation}\label{eq:initialgaplp}
\epsilon_0 \defeq 
\norm{\ma x_0 - b}_p^p \le r\norm{\ma x_0 - b}_2^p 
\le  \norm{\ma x^* - b}_2^p 
\le n^{(p-2)/2}f^*.
\end{equation}
The 
algorithm 
maintains the invariant
\[f(x_k) - f^*\le (2^{-p})^k \epsilon_0 \le (2^{-k}n)^p,\]
so that running $k=\log_2 \frac{n}{\delta^{1/p}}$ iterations guarantees 
multiplicative error of at most $\delta$\footnote{We note that $\log(n/\delta)$ iterations of our algorithm yield the stronger multiplicative accuracy guarantee of $\delta^p$, without an additional dependence on $p$.}.

Unlike the 
previous two applications, the function $g$ is \emph{not} QSC, as its 
Hessian is badly behaved near zero. Nevertheless we argue that an $\ell_2$ 
regularization of $g$ is QSC 
(Lemma~\ref{lem:lp_qsc}), and---because Algorithm~\ref{alg:ms-bacon} 
includes such regularization---the conclusion 
of the corollary still holds (Lemma~\ref{lem:progress}). The key to our analysis is showing that with each iteration the distance to 
the optimum $R$ shrinks (due to convergence to $x^*$) by the same factor 
that the QSC constant $M$ 
grows (due to diminishing regularization), such that $RM=O(p\sqrt{n})$ 
throughout, leading to the 
overall $\mathsf{poly}(p)n^{1/3}$ complexity guarantee.

\begin{algorithm}
\caption{High accuracy $\ell_p$ regression}
\begin{algorithmic}[1]
\label{alg:lp_reg}

\STATE{\textbf{Input:}  $\ma \in \R^{n \times d}, b \in \R^n$, 
multiplicative error tolerance $\delta \geq 0$.}
\STATE{Set $x_0 = \ma^\dagger b$ and $\epsilon_0=f(x_0) = \norm{\ma 
x_0 -b}_p^p$.}
\FOR{$k \le \log_2(n/\delta^{1/p})$}
\STATE{$\epsilon_k \gets 2^{-p} \epsilon_{k - 1}$}
\STATE{$x_k \gets $ output of Algorithm~\ref{alg:ms-bacon} applied on 
$f(x) = \norm{\ma x - b}_p^p$ with initialization $x_{k-1}$, desired 
accuracy $\epsilon_k$ and parameters 
$R=O(n^{{(p-2)}/{2p}}\epsilon_k^{{1}/{p}})$ 
and 
$M=O(p\sqrt{n}/R)$ 
(see Lemma~\ref{lem:progress})}\label{line:call-ms-bacon}
\ENDFOR
\end{algorithmic}
\end{algorithm}

We first bound the QSC of  $\ell_2$ regularization of $g$.

\begin{restatable}{lem}{lpqsc}\label{lem:lp_qsc}
For any $b \in \R^n$, $y \in \R^d$, $p \geq 3$, $\mu \geq 0$, the 
function $g(x) + \mu \norm{x-y}_2^2$ is $O(p 
\mu^{-1/(p-2)})$-QSC with respect to $\ell_2$.
\end{restatable}
\noindent
We next show approximate minimizers of $f$ are close to $x^*$.
\begin{restatable}{lem}{lpdist}
\label{lem:lp_dist}
For $x \in \R^d$ with $f(x) -f^* \le \epsilon$, we have $\norm{x - 
x^*}_\mm^p \leq 2^p n^{\frac{p-2}{2}} \epsilon$. 
\end{restatable}
\noindent
Finally, we bound the complexity of executions of 
Line~\ref{line:call-ms-bacon}.

\begin{restatable}{lem}{lpprogress}
\label{lem:progress}
Let $\epsilon_{k-1}\geq \delta f^*$. Initialized at $x_{k - 1}$ satisfying 
$f(x_{k - 
1}) 
-f^* \leq \epsilon_{k-1}$, Algorithm~\ref{alg:ms-bacon} 
computes $x_{k}$ with $f(x_{k}) -f^* \leq 2^{-p}\epsilon_{k-1} = 
\epsilon_k$ in 
$O(p^{14/3} 
n^{1/3} \log^3(n/\delta))$ linear system solves in $\ma^\top \md \ma$ 
for 
diagonal matrix $\md \succeq 0$. 
\end{restatable}
\noindent
We defer proofs of these statements to  Appendix~\ref{ssec:lpproofs}. Our 
final runtime follows from Lemma~\ref{lem:progress} and the fact that the 
loop in Algorithm~\ref{alg:lp_reg} repeats $O(\log\frac{n}{\delta})$ times.
\begin{cor}
Algorithm~\ref{alg:lp_reg} computes $x \in \R^d$ with
\[
\norm{\ma x - b}_p^p \leq (1+\delta) \norm{\ma x^* - b}_p^p 
\]
using $O(p^{14/3} n^{1/3} \log^4(n/\delta))$ linear system solves in 
$\ma^\top \md \ma$ for diagonal matrix $\md \succeq 0$. 
\end{cor}

\section{Lower bound}\label{sec:lb}
\newcommand{\prog}{i_r^{+}}
\newcommand{\alg}{\mathcal{A}}
\newcommand{\andd}{\,,\,}

In this section we establish a lower bound showing that the $(R/r)^{2/3}$ 
scaling in the oracle complexity we achieve is tight. For simplicity, we   
focus on a setting where the functions are defined on a bounded domain of 
radius $R>0$, and are 1-Lipschitz but potentially non-smooth; afterwards, 
we explain how to extend the result to unconstrained, differentiable and 
strictly convex functions. We assume throughout the section that 
$\mm = 
\mI$, i.e., that we work in the standard $\ell_2$ norm. We defer all the 
proofs in this section to \Cref{sec:lb-appendix}.

Following the literature on information-based 
complexity~\cite{nemirovski1983problem}, we state and prove our lower 
bound for the class of \emph{$r$-local oracles}, which for every query 
point $\bar{x}$ return a \emph{function} $f_{\bar{x}}$ that is identical to 
$f$ in a  
neighborhood of $x$. However, we additionally require the radius of this 
neighborhood to be at least $r$. Therefore, a query to an $r$-local oracle 
suffices to implement a ball optimization oracle (as well as a gradient 
oracle), and consequently a lower bound on algorithms interacting with an  
$r$-local oracles is also a lower bounds for algorithms a utilizing ball 
optimization oracle. The formal definition of the oracle class follows. 

\begin{defn}[Local oracles and algorithms]
	\label{def:loc-oralce}  We call $\oracleLoc$ an $r$-local oracle for 
		function $f:\domain \to \R$ if given query point $\bx \in \R^d$ it 
		returns 
		$f_{\bx}: \domain \to \R$ such that $f_{\bx}(x) = f(x)$ for all $x\in 
		\ballset$. We call (possibly randomized) algorithms that interact with 
		$r$-local oracles \emph{$r$-local algorithms}.
\end{defn}

\newcommand{\mprog}{p}
\newcommand{\algis}{a_i}
\newcommand{\algif}{\tilde{a}_i}
We prove our lower bound using a small extension of the well-established 
machinery of high-dimensional optimization lower 
bounds~\cite{nemirovski1983problem,woodworth2016tight,
	carmon2019lower_i,bubeck2019complexity}. To describe it, we 
start with the notion of coordinate progress, denoting for any $x\in\R^d$
\begin{equation}\label{eq:prog-def}
\prog(x) \defeq \min \{ i\in[d] \mid |x_j|  \le r~\mbox{for all}~j\ge i \},
\end{equation}
where we let $\prog(x) \defeq d+1$ when $|x_d| > r$, i.e. $\prog(x)$ is 
the index following the last ``large'' entry of $x$. With this notation, 
we define a key notion for proving our lower bound.
\begin{defn}[Robust zero-chains]
	\label{def:robust-chain}  Function $f:B_1(0) \to \R$ is an $r$-robust 
	zero-chain if $\forall \bx\in\R^d$, $x\in\ballset$, 
	\begin{equation*}
	f(x) = f(x_1, \ldots, x_{\prog(\bx)}, 0, \ldots, 0).
	\end{equation*}
\end{defn}
The notion of $r$-robust zero-chain we use here is very close to 
the robust zero-chain defined in~\citep[Definition 4]{carmon2019lower_i}, 
except here we require the equality to hold in a fixed ball rather than just a  
neighborhood of $\bx$. The following lemma shows that $r$-local 
algorithms operating on a random rotation of an $r$-robust zero-chain 
make slow progress with high probability.\footnote{
In \Cref{sec:lb-appendix} we provide a concise proof for 
\Cref{lem:robust-zero-chain-prog-bound} and compare it to existing 
proofs in the literature.
}
\begin{restatable}{lem}{proglem}\label{lem:robust-zero-chain-prog-bound}
Let $\frac{r}{R},\delta \in (0,1)$, $N\in\N$ and $d \ge 
\ceil[\big]{N+\frac{20R^2}{r^2} 
\log 
	\frac{20NR^2}{\delta r^2}}$. Let $f:\domain\to\R$ be an $r$-robust 
	zero-chain and let $\mU\in\R^{d\times d}$ be a random orthogonal matrix and fix 
	an $r$-local algorithm $\alg$. With probability at least $1-\delta$ over the 
	draw of $\mU$, there exists an $r$-local oracle $\oracle$ for $f_\mU(x) 
	\defeq f(\mU^\top x)$ 
	such that the queries $x_1, x_2, \ldots$ of 
	$\alg$ interacting with $\oracle$ satisfy
	\begin{equation*}
	\prog(\mU^\top x_i) \le i ~~\mbox{for all}~~i\le N.
	\end{equation*} 
\end{restatable}

With Lemma~\ref{lem:robust-zero-chain-prog-bound} in hand, to prove the lower 
bound we need to construct an $r$-robust zero-chain 
function $f_{N,r}$ with the additional property that every $x$ with 
$\prog(x) \le N$ is significantly suboptimal. Fortunately, 
Nemirovski's function~\cite{nemirovski1983problem} 
satisfies these properties.
\begin{restatable}{lem}{nemifunc}\label{lem:nemirovski-func}
	Let $r>0$ and $N\in\N$. Define
	\begin{equation}\label{eq:nemi-func}
	f_{N,r}(x) \defeq \max_{i\in[N]} \{ x_i - 4r \cdot i\}
	\end{equation}
	\begin{enumerate}
		\item \label{item:robust-zc} The function $f_{N,r}$ is an $r$-robust zero-chain.  
		\item \label{item:suboptimality} For all $x\in\domain$ such that $\prog(x)\le N$, we have $f_{N,r}(x) - \inf_{z\in \domain}f_{N,r}(z)\ge \frac{R}{\sqrt{N}} - 4Nr$.
		\item \label{item:lipschitz} The function $f_{N,r}$ is convex and 1-Lipschitz.
	\end{enumerate}
\end{restatable}
\Cref{lem:nemirovski-func}.\ref{item:robust-zc} is the main technical 
novelty of the section, while the other parts are known 
and stated for completeness. Combining Lemmas 
\ref{lem:robust-zero-chain-prog-bound} and \ref{lem:nemirovski-func} 
with 
appropriate choices of $N$ and $d$ immediately gives the lower bound.

\begin{restatable}{thm}{lbmain}\label{thm:lb}
	Let $\frac{r}{R},\delta\in(0,1)$ and $d = \ceil[\big]{60 
		(\frac{R}{r})^2 \log 
		\frac{R}{\delta \cdot r}}$. There 
	exists a distribution $P$ over convex and 1-Lipschitz functions from 
	$\domain\to\R$ and corresponding $r$-local oracles such that the 
	following holds for any $r$-local algorithm.
	With probability at least $1-\delta$ over the draw of $(f,\oracle)\sim P$, 
	when the algorithm interacts with $\oracle$, its first 
	$\ceil[\big]{\frac{1}{10} (\frac{R}{r})^{2/3}}$ 
	queries are 
	at least $R^{2/3}r^{1/3}$ suboptimal for $f$. 
\end{restatable}
\begin{proof}%
	Set $N=\floor*{\frac{1}{10}(\frac{R}{r})^{2/3}}$ and $d\ge 
	\ceil*{\frac{60R^2}{r^2}\log\frac{R}{\delta r}} \ge 
	\ceil*{N+\frac{20R^2}{r^2}\log\frac{20NR^2}{\delta r^2}}$. Apply 
	\Cref{lem:robust-zero-chain-prog-bound} with 
	\Cref{lem:nemirovski-func}.\ref{item:robust-zc} to argue that for any $r$-local
	algorithm, with probability at least $1-\delta$ the first $N$ queries 
	$x_1,\ldots,x_N$ satisfy $\prog(\mU^\top x_i) \le N$, and substitute 
	into 
	\Cref{lem:nemirovski-func}.\ref{item:suboptimality} to conclude that the 
	suboptimality of each query is at least 
	$(\sqrt{10}-\frac{4}{10})(R^2 r)^{1/3} 
	\ge 
	(R^2 r)^{1/3}$.
\end{proof}

\Cref{thm:lb} shows as long as we wish to solve the minimization problem to 
accuracy $\epsilon = o(R^{2/3}r^{1/3})$, for any $r$-local algorithm, there is a 
function requiring $\Omega((R/r)^{2/3})$ queries to an $r$-local oracle, which gives strictly more information than a ball optimization oracle, proving our desired lower bound. However, our acceleration scheme assumes unconstrained, smooth and 
strictly convex problems. We now outline modifications to the 
construction~\eqref{eq:nemi-func} extending it to this regime.
	\paragraph{Unconstrained domain.} Following the approach of 
	\citet{diakonikolas2018lower}, we note that the construction 
	$f(x) = \max\{\half 
	f_{N,r}(x), \norm{x}-\frac{R}{2}\}$  provides a hard instance for 
	algorithms with unbounded queries, because any query with norm larger than
	$\norm{x}$ is uninformative about the rotation of coordinates and has a 
	positive function value, so that the minimizer is still constrained to a ball 
	of radius $R$.
	\paragraph{Smooth functions.} The smoothing argument 
	of~\citet{guzman2015lower} shows that $f(x) = \inf_{x'\in 
	\ball{r}(x)}\{f_{N,2r}(x')+\frac{1}{r}\norm{x'-x}^2\}$ is an $r$-robust 
	zero-chain 
	that is also $2/r$-smooth and satisfies $|f(x)-f_{N,2r}(x)|\le r$ for all 
	$x$. Consequently, the lower bound holds for $O(1/r)$ smooth functions.
	\paragraph{Strictly convex functions.} The function $f(x) = f_{N,r}(x) 
	+ \frac{r^{1/3}}{2R^{4/3}}\norm{x}^2$ provides an 
	$(r^{1/3}R^{-4/3})$-strongly convex hard instance, since we 
	can add the strongly convex regularizer directly in the local oracle 
	without revealing additional information, and the regularizer size 
	is small enough so as not to significantly affect the optimality gap.

\bibliographystyle{plainnat}
\bibliography{main.bib}

\appendix

\newpage

\part*{Supplementary material}

\section{Unaccelerated optimization with a ball optimization oracle}
\label{sec:vanilla}

Here, we state and analyze the unaccelerated algorithm
for optimization of convex function $f$ with access to a ball optimization
oracle. For simplicity of exposition, we assume that the oracle $\oracleBall$
is a $(0,r)$-oracle, i.e. is exact, and we perform our analysis in the $\ell_2$ norm; for a general Euclidean seminorm, a change of basis suffices to give the same guarantees.

\begin{algorithm}
\caption{Iterating ball optimization}
\begin{algorithmic}[1]
\label{alg:unaccel_framework}
\STATE{\textbf{Input: } Function $f:\R^{d}\rightarrow\R$ and initial point 
$x_0\in \R^d$.}

\FOR{$k=1,2,...$}

\STATE{$x_{k}\gets\oracleBall(x_{k-1})$}

\ENDFOR
\end{algorithmic}
\end{algorithm}
We first note that the distance $\norm{x_{k}-x^{*}}_{2}$ is decreasing
in $k$.
\begin{lem}
For all $x\in\R^{d}$, $\norm{\oracleBall(x)-x^{*}}_{2} \le 
\norm{x-x^{*}}_{2}$.\label{lem:closer-unaccel}
\begin{proof}
The claim is obvious if $\oracleBall(x)=x^{*}$, so we assume this
is not the case. Note that for any $\tilde{x}$ with $\norm{\tilde{x}-x}_{2}\le r$,
if there is any point $\hat{x}$ on the line between $\tilde{x}$
and $x^{*}$, then by strict convexity $f(\hat{x})<f(\tilde{x})$.
Now, clearly $\oracleBall(x)$ lies on the boundary of the ball around
$x$, and moreover the angle between the vectors $x-\oracleBall(x)$
and $x^{*}-\oracleBall(x)$ must be obtuse, else the line between
$\oracleBall(x)$ and $x^{*}$ intersects the ball twice. Thus, by
law of cosines $\norm{\oracleBall(x)-x^{*}}_{2}\le\sqrt{\norm{x-x^{*}}_{2}^{2}-r^{2}}$,
yielding the conclusion.
\end{proof}
\end{lem}
\begin{thm}
\label{thm:mainclaim-unaccel}
Suppose for some $x_0 \in \R^d$, $f(x_0) - f(x^*) \le \eps_0$ and $\norm{x_{0}-x^{*}}_{2}\le R$, where $x^*$ is the global minimizer of $f$. Algorithm~\ref{alg:unaccel_framework} computes an $\eps$-approximate minimizer in $O\left(\frac{R}{r}\log\frac{\eps_0}{\eps}\right)$ calls to $\oracleBall$.\label{thm:decrease-unaccel}
\end{thm}

\begin{proof}
Define $\tilde{x}_{k}\defeq\left(1-\frac{r}{R}\right)x_{k-1}+\frac{r}{R}x^{*}$,
and note that because $\norm{x_{k-1}-x^{*}}_{2}\le R$, $\tilde{x}_{k}$
is in the ball of radius $r$ around $x_{k-1}$. Thus, convexity yields
\[
f(x_{k})\le f(\tilde{x}_{k})\le\left(1-\frac{r}{R}\right)f(x_{k-1})+\frac{r}{R}f(x^{*}) \Rightarrow f(x_k)-f(x^*) \le \left(1-\frac{r}{R}\right)\left(f(x_{k-1})-f(x^*)\right).
\]
Iteratively applying this inequality yields the conclusion.
\end{proof}

\section{Analysis of Monteiro-Svaiter acceleration}
\label{sec:framework}

In this section, we prove Proposition~\ref{prop:ball_constrained_error_bound}.
We do so by first proving a sequence of lemmas demonstrating properties 
of Algorithm~\ref{alg:msaccel_framework}. Throughout, we recall $\nabla 
f(x) \in \text{Im}(\mm)$ for all $x$ by assumption. We note that these are 
variants of existing bounds in the 
literature~\citep[e.g.][]{MonteiroS13a,bubeck2019complexity}.
\begin{lem}
\label{lem:lam_a_bound} For all $k\ge0$,
\[
\lambda_{k+1}A_{k+1}=a_{k+1}^{2}\text{ and }\sqrt{A_{k}}\ge\half\sum_{i\in[k]}\sqrt{\lambda_{i}}.
\]
\end{lem}

\begin{proof}
The first claim is from solving a quadratic in the definition
of $a_{k+1}$. The second follows from
\begin{align*}
\sqrt{A_{k}}\geq\sqrt{A_{k}}-\sqrt{A_{0}}&=\sum_{i\in[k]}\left(\sqrt{A_{i}}-\sqrt{A_{i-1}}\right)=\sum_{i\in[k]}\frac{a_{i}}{\sqrt{A_{i}}+\sqrt{A_{i-1}}}\\
&=\sum_{i\in[k]}\frac{\sqrt{\lambda_{i}A_{i}}}{\sqrt{A_{i}}+\sqrt{A_{i-1}}}\geq\frac{1}{2}\sum_{i\in[k]}\sqrt{\lambda_{i}}
\end{align*}
where we used that $A_0 \geq 0$ and $\{A_i\}$ are increasing.
\end{proof}

\begin{lem}
\label{lem:step_lower_bound} For all $k\ge0$, if  $\normm{x_{k+1}-y_{k}}>0$, we have for $\sigma\in[0,1)$,
\[
\normmd{\grad f(x_{k+1})}>0\text{ and }\lambda_{k+1}\geq\frac{\normm{x_{k+1}-y_{k}}}{\normmd{\grad f(x_{k+1})}}\left(1-\sigma\right) > 0\,.
\]
\end{lem}
\begin{proof}
For the first claim, by \eqref{eq:ms_requirement},
\[
\|x_{k+1}-y_k\|_{\mm}-\lambda_{k+1}\|\nabla f(x_{k+1})\|_{\mm^\dagger} \le \norm{x_{k+1}-(y_{k}-\lambda_{k+1}\mmd\nabla f(x_{k+1}))}_{\mm}\le\sigma\norm{x_{k+1}-y_{k}}_{\mm},
\]
since by assumption, for some $\sigma\in[0,1)$, $\normm{x_{k+1}-y_{k}}>0$, therefore
$\norm{\grad f(x_{k+1})}_{\mmd} = 0$ would contradict this assumption. 

For the second claim, Cauchy-Schwarz gives
\begin{align*}
\sigma^{2}\normm{x_{k+1}-y_{k}}^{2} & \geq\normm{x_{k+1}-\left(y_{k}-\lambda_{k+1}\mmd\grad f(x_{k+1})\right)}^{2}\\
 &\geq\norm{x_{k+1}-y_{k}}_{\mm}^{2}-2\lambda_{k+1}\norm{\grad f(x_{k+1})}_{\mmd}\norm{x_{k+1}-y_{k}}_{\mm}+\lambda_{k+1}^{2}\norm{\grad f(x_{k+1})}_{\mmd}^{2}.
\end{align*}
Solving the quadratic in $\lambda_{k+1}$ implies, for $P\defeq\norm{\grad f(x_{k+1})}_{\mmd}\norm{x_{k+1}-y_{k}}_{\mm}$,
\begin{align*}
\lambda_{k+1} & \geq\frac{2P-\sqrt{4P^{2}-4(1-\sigma^{2})P^{2}}}{2\norm{\grad f(x_{k+1})}_{\mmd}^{2}}=\frac{P(1-\sigma)}{\norm{\grad f(x_{k+1})}_{\mmd}^{2}}=\frac{\norm{x_{k+1}-y_{k}}_{\mm}}{\norm{\grad f(x_{k+1})}_{\mmd}}\left(1-\sigma\right).
\end{align*}
\end{proof}
Next, we provide the following lemma which gives a recursive bound
for the potential, $p_{k}$, which we define as follows:
\[p_k \defeq A_k \eps_k + r_k, \text{ where } \eps_k \defeq f(x_k) - f(x^*),\; r_k \defeq \half\normm{v_k - x^*}^2.\] We remark that the proof does not use
(\ref{eq:ms_requirement}) beyond using the property that $a_{k+1}>0$
(regardless of how they are induced by $\lambda_{k+1}$).
\begin{lem}
\label{lem:accel_gen} For all $k\ge0$,
\[
p_{k+1}\leq p_{k}+\frac{A_{k+1}^{2}}{2a_{k+1}^{2}}\left(\normm{x_{k+1}-\left(y_{k}-\frac{a_{k+1}^{2}}{A_{k+1}}\mmd\grad f(x_{k+1})\right)}^{2}-\normm{x_{k+1}-y_{k}}^{2}\right).
\]	
\end{lem}

\begin{proof}
By Lemma~\ref{lem:step_lower_bound} we have that $\lambda_{k+1}>0$,
so that $a_{k+1}>0$. Then,
\[
v_{k}=\frac{1}{a_{k+1}}\left(A_{k+1}y_{k}-A_{k}x_{k}\right)=x_{k+1}+\frac{A_{k+1}}{a_{k+1}}(y_{k}-x_{k+1})+\frac{A_{k}}{a_{k+1}}(x_{k+1}-x_{k}).
\]
Consequently, convexity of $f$, i.e., $\innerProduct{\grad f(b)}{a-b}\leq f(a)-f(b)$
for all $a,b\in\R^{n}$, yields
\[
a_{k+1}\innerProduct{\grad f(x_{k+1})}{x^{*}-v_{k}}\leq A_{k}\epsilon_{k}-A_{k+1}\epsilon_{k+1}+A_{k+1}\innerProduct{\grad f(x_{k+1})}{x_{k+1}-y_{k}}.
\]
Further, expanding $r_{k+1}=\frac{1}{2}\normm{v_{k+1}-x_{*}}^{2}$,
where we recall $v_{k+1}=v_{k}-a_{k+1}\mmd\nabla f(x_{k+1})$, gives
\begin{align*}
\frac{1}{2}\normm{v_{k+1}-x^{*}}^{2} & =r_{k}+\frac{a_{k+1}^{2}}{2}\normmd{\grad f(x_{k+1})}^{2}+a_{k+1}\innerProduct{\mm\mmd\grad f(x_{k+1})}{x^{*}-v_{k}}.
\end{align*}
Combining these inequalities, and recalling $\mm \mmd \nabla f(x_{k + 1}) = \nabla f(x_{k + 1})$, then yields that
\[
A_{k+1}\epsilon_{k+1}+r_{k+1}\leq A_{k}\epsilon_{k}+r_{k}+\frac{a_{k+1}^{2}}{2}\normmd{\grad f(x_{k+1})}^{2}+A_{k+1}\innerProduct{\grad f(x_{k+1})}{x_{k+1}-y_{k}}.
\]
The result then follows from $p_{k}=A_{k}\epsilon_{k}+r_{k}$ and
the fact that 
\begin{align*}
\frac{A_{k+1}^{2}}{2a_{k+1}^{2}}\normFull{x_{k+1}-\left(y_{k}-\frac{a_{k+1}^{2}}{A_{k+1}}\mmd\grad f(x_{k+1})\right)}_{\mm}^{2}\\
=\frac{A_{k+1}^{2}}{2a_{k+1}^{2}}\normm{x_{k+1}-y_{k}}^{2}+A_{k+1}\innerProduct{\grad f(x_{k+1})}{x_{k+1}-y_{k}}+\frac{a_{k+1}^{2}}{2}\normmd{\grad f(x_{k+1})}^{2}.
\end{align*}
\end{proof}
Next, we use (\ref{eq:ms_requirement}) and the choice of $a_{k+1}$
in the algorithm to improve the bound in Lemma~\ref{lem:accel_gen}.
\begin{lem}
\label{lem:accel_framework} For all $k\geq0$, 
\[
p_{k}+\sum_{i\in[k]}\frac{(1-\sigma^{2})A_{i}}{2\lambda_{i}}\normm{x_{i+1}-y_{i}}^{2}\leq p_{0}\,.
\]
\end{lem}
\begin{proof}
Lemma \ref{lem:lam_a_bound} gives that for our choice of parameters,
$\lambda_{k+1}A_{k+1}=a_{k+1}^{2}$ for all $k\geq0$. Lemma~\ref{lem:accel_gen}
then implies that 
\begin{align*}
A_{k+1}\epsilon_{k+1}+r_{k+1}&\leq A_{k}\epsilon_{k}+r_{k}+\frac{A_{k+1}}{2\lambda_{k+1}}\left(\normm{x_{k+1}-\left(y_{k}-\lambda_{k+1}\mmd\grad f(x_{k+1})\right)}^{2}-\normm{x_{k+1}-y_{k}}^{2}\right)\\
&\leq A_{k}\epsilon_{k}+r_{k}+\frac{(\sigma^2-1)A_{k+1}}{2\lambda_{k+1}}\normm{x_{k+1}-y_{k}}^{2}
\end{align*}
where we used \eqref{eq:ms_requirement} and the claim now follows from inductively applying
the resulting bound.
\end{proof}
Below we give a diameter bound on the iterates from the algorithm.
\begin{lem}
\label{lem:domain_bound} If $x_{0}=v_{0}$, then for all $k\geq0$
we have
\[
\normm{x_{k}-x^{*}}\leq\frac{2-\sigma}{1-\sigma}\sqrt{2p_{0}},\;\normm{v_{k}-x^{*}}\le\sqrt{2p_{0}}.
\]
\end{lem}

\begin{proof}
Since $p_{k}=A_{k}\epsilon_{k}+r_{k}$, the second claim follows immediately
from Lemma~\ref{lem:accel_framework} implying that $\half\normm{v_{k}-x^{*}}^{2}=r_{k}\leq p_{0}$
for all $k\geq0$. Further, convexity and the triangle inequality
imply that 
\begin{align*}
\normm{x_{k+1}-x^{*}} & \leq\normm{y_{k}-x^{*}}+\normm{x_{k+1}-y_{k}}\\
 & \leq\frac{A_{k}}{A_{k+1}}\normm{x_{k}-x^{*}}+\frac{a_{k+1}}{A_{k+1}}\normm{v_{k}-x^{*}}+\normm{x_{k+1}-y_{k}}.
\end{align*}
Rearranging and applying recursively yields that 
\begin{align*}
A_{k+1}\normm{x_{k+1}-x^{*}} & \leq A_{k}\normm{x_{k}-x^{*}}+a_{k+1}\normm{v_{k}-x^{*}}+A_{k+1}\normm{x_{k+1}-y_{k}}\\
 & \leq A_{0}\normm{x_{0}-x^{*}}+\sum_{i=0}^{k}a_{i+1}\normm{v_{i}-x^{*}}+\sum_{i=0}^{k}A_{i+1}\normm{x_{i+1}-y_{i}}.
\end{align*}
Now, using $A_{k+1}=A_{0}+\sum_{i=0}^{k}a_{i+1}$, $x_{0}=v_{0}$,
the previously-derived $\normm{v_{i}-x^{*}}\leq\sqrt{2p_{0}}$,
and Cauchy-Schwarz,
\[
\normm{x_{k+1}-x_{*}}\leq\sqrt{2p_{0}}+\frac{1}{A_{k+1}}\sqrt{\left(\sum_{i=0}^{k}\lambda_{i+1}A_{i+1}\right)\left(\sum_{i=0}^{k}\frac{A_{i+1}}{\lambda_{i+1}}\normm{x_{i+1}-y_{i}}^{2}\right).}
\]
Now, since $\lambda_{k+1}A_{k+1}=a_{k+1}^{2}$ and $\sqrt{a+b}\leq\sqrt{a}+\sqrt{b}$
for nonnegative $a,b$ we have
\[
\sqrt{\sum_{i=0}^{k}\lambda_{i+1}A_{i+1}}\leq\sum_{i=0}^{k}\sqrt{\lambda_{i+1}A_{i+1}}=\sum_{i=0}^{k}a_{i+1}=A_{k+1},
\]
and the result follows from 
\[\sum_{i=0}^{k}\frac{A_{i+1}}{\lambda_{i+1}}\normm{x_{i+1}-y_{i}}^{2}\leq(1-\sigma^{2})^{-1}2p_{0}\]
(due to Lemma \ref{lem:accel_framework}), and $\sqrt{(1-\sigma^{2})^{-1}}\leq(1-\sigma)^{-1}$. 
\end{proof}

We next give a basic helper lemma which will be useful in the proof of Proposition~\ref{prop:ball_constrained_error_bound}.
\begin{lem}
	\label{lem:growth_lemma} Let $\left\{ B_{k}\right\} _{k\in\mathbb{N}}$
	be a nonnegative, nondecreasing sequence such that $B_{k}\geq\sum_{i\in[k]}\alpha B_{i}$
	for some $\alpha\in[0,1)$ and all $k$. Then for all $k$, $B_{k}\geq\exp(\alpha(k-1))B_{1}$.
\end{lem}

\begin{proof}
	Extend $C(t)\defeq B_{\lceil t\rceil}$ for 
	all $t\ge1$, and let
	$C(t)\defeq\exp(\alpha(t-1))B_{1}$ for $t\in[0,1]$. Then for all
	$t\ge1$, 
	
	\[
	C(t)=B_{\lceil t\rceil}\ge\alpha\sum_{i\in[\lceil t\rceil]}B_{i}\ge\alpha\int_{0}^{t}C(s)ds,
	\]
	and it is easy to check that this inequality holds with equality for
	$t\in[0,1]$ as well. Letting $L(t)$ solve this integral inequality,
	i.e., $L(t)=C(t)$ for $t\in[0,1]$ and
	\[
	L(t)=\alpha\int_{0}^{t}L(s)ds,
	\]
	$L(t)=\exp(\alpha(t-1))C(1)$, and inequality
	$C(t)\ge L(t)$ yields the claim, recalling $B_k = C(k)$ for $k \in \N$.
\end{proof}

Now we are ready to put everything together and prove the main result of this section.
\convergencerateaccel*
\begin{proof}
First, we will show the bound
\begin{equation}
\label{eq:simplerclaim}
f(x_k) - f(x^*) \le \frac{p_0}{A_1}\exp\left(-\frac{3}{2}\left(\frac{r(1 - \sigma)}{\sqrt{p_0}}\right)^{2/3}(k - 1)\right).
\end{equation}
The reverse H\"older inequality with $p=3/2$ states that for all $u,v\in\R_{>0}^{k}$,
\begin{equation}
\innerProduct uv\geq\left(\sum_{i\in[k]}u_{i}^{2/3}\right)^{3/2}\cdot\left(\sum_{i\in[k]}v_{i}^{-2}\right)^{-1/2}.\label{eq:reverseholder}
\end{equation}
Lemma~\ref{lem:lam_a_bound} gives $\sqrt{A_{k}}\geq\frac{1}{2}\sum_{i\in[k]}\sqrt{\lambda_{i}}$.
Moreover, $\norm{x_{i}-y_{i-1}}_{\mm}>0$ by the assumptions of this
proposition, which implies by Lemma~\ref{lem:step_lower_bound}
that $A_{i}\ge\lambda_{i}>0$ as well. Thus, we can apply (\ref{eq:reverseholder})
with $u_{i}=\sqrt{A_{i}}\norm{x_{i}-y_{i-1}}_{\mm}$ and $v_{i}=\sqrt{\lambda_{i}}/u_{i}$, yielding
\begin{equation}
\sqrt{A_{k}}\geq\frac{1}{2}\sum_{i\in[k]}\sqrt{\lambda_{i}}\geq\frac{1}{2}\left(\sum_{i\in[k]}\left(\sqrt{A_{i}}\norm{x_{i}-y_{i-1}}_{\mm}\right)^{2/3}\right)^{3/2}\left(\sum_{i\in[k]}\left(\frac{\sqrt{\lambda_{i}}}{\sqrt{A_{i}}\norm{x_{i}-y_{i-1}}_{\mm}}\right)^{-2}\right)^{-1/2}.\label{eq:accel_prog1}
\end{equation}
Applying
Lemma~\ref{lem:accel_framework} yields that 
\begin{equation}
\sum_{i\in[k]}\left(\frac{\sqrt{\lambda_{i}}}{\sqrt{A_{i}}\norm{x_{i}-y_{i-1}}_{\mm}}\right)^{-2}=\sum_{i\in[k]}\frac{A_{i}\norm{x_{i}-y_{i-1}}_{\mm}^{2}}{\lambda_{i}}\leq\left(\frac{2}{1-\sigma^{2}}\right)p_{0}.\label{eq:accel_prog2}
\end{equation}
Now, since $\norm{x_{i}-y_{i-1}}_{\mm}\geq r$ by assumption, combining 
(\ref{eq:accel_prog1}) and (\ref{eq:accel_prog2}) gives
\[
A_{k}^{1/3}\geq\left(\frac{1}{2}\right)^{2/3}\left(\sum_{i\in[k]}A_{i}^{1/3}r^{2/3}\right)\left(\left(\frac{2}{1-\sigma^{2}}\right)p_{0}\right)^{-1/3}=\sum_{i\in[k]}A_{i}^{1/3}\left(\frac{r^{2}(1-\sigma^{2})}{8p_{0}}\right)^{1/3}.
\]
Finally, applying Lemma~\ref{lem:growth_lemma} implies that for
all $k\ge0$
\[
A_{k}\geq\exp\left(\frac{3}{2}\left(\frac{r^{2}(1-\sigma^{2})}{p_{0}}\right)^{1/3}(k-1)\right)A_{1}.
\]
Now, \eqref{eq:simplerclaim} follows from $\epsilon_k \le p_k / A_k \le p_0 /A_k$ (we have 
$p_k \le p_0$ from Lemma~\ref{lem:accel_framework}) 
and $(1-\sigma^{2})\leq(1-\sigma)^{2}$. Now, by our choice of $A_{0} = R^2/2\eps_0$, 
we have $p_{0} = R^{2}$. As $A_{1}\geq A_{0}$,
\[
\frac{p_{0}}{A_{1}}\le\frac{R^{2}}{A_{0}} =2\eps_{0}.
\]
Combining these bounds in the context of \eqref{eq:simplerclaim}, and using $3/2 > 1$, yields the result.
\end{proof}

\section{MS oracle implementation proofs}
\label{sec:lslipschitz}

First, we prove our characterization of the optimizer of a ball-constrained 
problem. 

\balloptchar*
\begin{proof}
	By considering the optimality conditions of the Lagrange dual problem
	\[
	\min_z \max_{\lambda \ge 0} f(z) + \frac{\lambda}{2}\left(\normm{z - y}^2 - r^2\right),
	\]
	we see there is some $\lambda\geq 0$ such that
	\[
	\grad f(z)=-\lambda\grad_{z}\left(\frac{1}{2}\normm{z - y}^{2}-\frac{r^2}{2}\right)=-\lambda \mm(z - y)\,.
	\]
	If $\lambda=0$ then $\grad f(z)=0$ and $z$ is a minimizer of $f$. On the other hand, if $\lambda>0$, then $\norm{z - y}_{\mm}=r$ and $\grad f(z) = -\lambda \mm(z - y)$. By taking the $\mmd$ seminorm of both sides of this condition, $\normmd{\grad f(z)} = \lambda \normm{z - y} = \lambda r$; solving for $\lambda$ and substituting yields the result.
\end{proof}

Next, on the path to proving Proposition~\ref{prop:line_search}, we give a helper result which bounds the change in the solution to a ball-constrained problem as we move the center.
\begin{lem}
\label{lem:lip_bound} For strictly convex, twice differentiable $f:\R^{d}\rightarrow\R$, let $\mm$ be a positive semidefinite matrix where $\grad f(u) \in \textup{Im}(\mm)$ for all $u \in \R^d$. Let $x,v\in\R^{d}$ be arbitrary vectors, and for all $t\in[0,1]$, let
\[
y_{t}\defeq tx+(1-t)v,\;z_{t}\defeq\argmin_{z \in \ballsett{y_t}}f(z).
\]
Then, for all $t\in[0,1]$ we have
\[
\normFull{\frac{d}{dt}z_{t}}_{\hess f(z_{t})}=\normFull{\frac{d}{dt}\grad f(z_{t})}_{(\hess f(z_{t}))^{-1}}\leq\norm{x-v}_{\hess f(z_{t})}.
\]
\end{lem}
\begin{proof}
Let $t\in[0,1]$ be arbitrary. If $\norm{z_{t}-y_{t}}_{\mm}<r$, then
$z_{t}$ is the minimizer of $f$, i.e. $\grad f(z_{t})=0$ and $\frac{d}{dt}z_{t}=0$
yielding the result (as in this case the minimizer stays in the interior
for small perturbations of $y_{t}$). For the remainder
of the proof assume that $\norm{z_{t}-y_{t}}_{\mm}=r$, in which
case Lemma~\ref{lem:opt_characterization} yields that
\begin{equation}
\grad f(z_{t})=-\frac{\norm{\grad f(z_{t})}_{\mm^\dagger}}{r} \mm (z_{t}-y_{t})\,.\label{eq:lip_bound_1}
\end{equation}
Now, differentiating both sides with respect to $t$ yields that 
\begin{align}
\frac{d}{dt}\left(\grad f(z_{t})\right) & =-\frac{\innerProduct{\grad f(z_{t})}{ \mm^\dagger \frac{d}{dt}\left(\grad f(z_{t})\right)}}{r\norm{\grad f(z_{t})}_{\mm^\dagger}} \mm (z_{t}-y_{t})-\frac{1}{r}\norm{\grad f(z_{t})}_{\mm^\dagger} \mm \left(\frac{d}{dt}z_{t}-\left(x-v\right)\right).\label{eq:lip_bound_2}
\end{align}
Combining (\ref{eq:lip_bound_1}) and (\ref{eq:lip_bound_2}) and
taking an inner product of both sides with $\mm^\dagger \frac{d}{dt}(\grad f(z_{t}))$
yields that
\begin{align*}
\normFull{\frac{d}{dt}\left(\grad f(z_{t})\right)}_{\mm^\dagger}^{2} & =\frac{\innerProduct{\grad f(z_{t})}{\mm^\dagger \frac{d}{dt}\left(\grad f(z_{t})\right)}^{2}}{\norm{\grad f(z_{t})}_{\mm^\dagger}^{2}}-\frac{1}{r}\norm{\grad f(z_{t})}_{\mm^\dagger}\innerProduct{\frac{d}{dt}z_{t}-(x-v)}{\mm \mm^\dagger \frac{d}{dt}(\nabla f(z_{t}))}.
\end{align*}
Next, Cauchy-Schwarz implies $\innerProduct{\grad f(z_{t})}{\mm^\dagger \frac{d}{dt}\left(\grad f(z_{t})\right)}^{2}\leq\norm{\grad f(z_{t})}_{\mm^\dagger}^{2}\cdot\norm{\frac{d}{dt}(\grad f(z_{t}))}_{\mm^\dagger}^{2}$
, so the first two terms in the above display cancel. Rearranging
the last term yields
\[
\innerProduct{\frac{d}{dt}z_{t}}{\mm \mm^\dagger \frac{d}{dt}(\grad f(z_{t}))}\leq\innerProduct{x-v}{ \mm \mm^\dagger \frac{d}{dt}(\nabla f(z_{t}))}.
\]
Since $\grad f(z_{t})$ is in the image of $\mm$ for all $t$, $\frac{d}{dt}(\grad f(z_{t}))$ must also be in the image of $\mm$. Thus, we can drop the $\mm \mm^\dagger$ matrices from the above expression. Also as $\frac{d}{dt}(\grad f(z_{t}))=\hess f(z_{t})\frac{d}{dt}z_{t}$
, this simplifies to
\[
\normFull{\frac{d}{dt}z_{t}}_{\hess f(z_{t})}^{2}\leq\innerProduct{x-v}{\hess f(z_{t})\frac{d}{dt}z_{t}}\leq\normFull{\frac{d}{dt}z_{t}}_{\hess f(z_{t})}\cdot\norm{x-v}_{\hess f(z_{t})}.
\]
Dividing both sides by $\norm{\frac{d}{dt}z_{t}}_{\hess f(z_{t})}$
and applying $\frac{d}{dt}\grad f(z_{t})=\hess f(z_{t})\frac{d}{dt}z_{t}$
then yields the result.
\end{proof}

We now bound the Lipschitz constant of the function $g(\lambda) = \lambda\normmd{\nabla f(\ztl)}$, where we recall the definitions
\begin{equation}\label{eq:tlamdef}\begin{aligned}a_\lambda \defeq \frac{\lambda + \sqrt{\lam^2 + 4\lam A}}{2},\; t_\lam \defeq \frac{A}{A + a_\lam},\; \ytl \defeq t_\lam x + (1 - t_\lam) v,\; \ztl \defeq \min_{z\in\ballsett{\ytl}} f(z).\end{aligned}\end{equation}
\begin{lem}
	\label{lem:g-lip} Let $f$ be $L$-smooth in $\normm{\cdot}$. Assume that in \eqref{eq:tlamdef}, $\normm{x - x^*} \le D$ and $\normm{v - x^*} \le D$. For all $\lambda\ge0$,
	\[
	\left|\frac{d}{d\lambda}g(\lambda)\right|\le L(2D + r).
	\]
\end{lem}
\begin{proof}
We compute
\begin{align}
\frac{d}{d\lambda}g(\lambda) & =\norm{\grad f(\ztl)}_{\mm^\dagger}+\lambda\frac{\innerProduct{\grad f(\ztl)}{\mm^\dagger \nabla^2 f(\ztl)\left(\frac{d}{dt_{\lambda}}\ztl\right)}}{\norm{\grad f(\ztl)}_{\mm^\dagger}}\frac{d}{d\lambda}t_{\lambda}.\label{eq:needboundlip}
\end{align}
First, direct calculation yields
\[
\frac{d}{d\lambda}t_{\lambda}=-\frac{A}{(A+a_\lambda)^{2}}\frac{d}{d\lambda}a_{\lambda}=-\frac{A}{(A+a_\lambda)^2}\cdot\frac{1}{2}\left(1+ \left(\lambda^{2}+4A\lam\right)^{-1/2}\left(\lambda+2A\right)\right).
\]
Consequently, recalling the definition of $a_\lam$,
\begin{equation}
\begin{aligned}
\left|\lambda\frac{d}{d\lambda}t_{\lambda}\right| & = \left|\frac{2 A \lambda}{(2A + \lambda + \sqrt{\lambda^2 + 4 A \lambda})^2} \left(1 + \frac{\lambda+2A}{\sqrt{\lambda^2 + 4 A\lam}} \right)  \right| \\
& = \left|\frac{2 A \lambda}{(2A + \lambda + \sqrt{\lambda^2 + 4 A \lambda})\sqrt{\lambda^2 + 4 A\lam}} \right| \leq \frac{2 A \lambda}{\lambda^2 + 4 A\lam} \leq \frac{1}{2} .\label{eq:bound1lip}
\end{aligned}
\end{equation}
where we used that $A, \lambda > 0$.
Next, by triangle inequality and smoothness in the $\mm$-norm,
\begin{equation}
\norm{\nabla f(\ztl)}_{\mm^\dagger}\le L\norm{\ztl-x^{*}}_{\mm}\le L\left(\norm{\ztl-\ytl}_{\mm}+\norm{\ytl-x^{*}}_{\mm}\right)\le L(r+D).\label{eq:bound2lip}
\end{equation}
In the last inequality, we used convexity of norms and $\norm{x-x^{*}}_{\mm},\;\norm{v-x^{*}}_{\mm}\le D$.
The final bound we require is due to Lemma \ref{lem:lip_bound}: observe
\begin{equation}\begin{aligned}
\innerProduct{\nabla f(\ztl) \mm^\dagger}{\nabla^2 f(\ztl)\left(\frac{d}{dt_\lambda}\ztl\right)} &\le \norm{\grad f(\ztl)}_{\mm ^\dagger \nabla^2 f(\ztl) \mm^\dagger }\normFull{\frac{d}{dt_\lambda}\ztl}_{\nabla^2 f(\ztl)} \\
\le \sqrt{L} \norm{\nabla f(\ztl)}_{\mm^\dagger} \norm{x - v}_{\nabla^2 f(\ztl)} &\le L\norm{\nabla f(\ztl)}_{\mm^\dagger}\normm{x - v} \le 2 LD \norm{\nabla f(\ztl)}_{\mm^\dagger}.\label{eq:bound3lip}
\end{aligned}\end{equation}
The first inequality is by Cauchy-Schwarz, the second is due to Lemma \ref{lem:lip_bound} and $\mm^\dagger \nabla^2 f(\ztl) \mm^\dagger \preceq L \mm^\dagger$ by smoothness, and the third is again from smoothness with $\nabla^2 f(\ztl) \preceq L \mm$. Combining (\ref{eq:needboundlip}),
(\ref{eq:bound1lip}), (\ref{eq:bound2lip}), and (\ref{eq:bound3lip})
yields the claim. 
\end{proof}
We now prove Proposition~\ref{prop:line_search}.
\msoracleguarantee*
\begin{proof}
	This proof will require three bounds on the size of the parameter $\delta$ used in the ball 
	optimization oracle. We state them here, and show that the third implies the 
	other two. We require
	\begin{equation}\label{eq:deltabounds}\delta \le \min\left\{\frac{\eps}{2L(D + r)},\; \sqrt{\frac{2\eps}{L}},\; \frac{r}{12\left(1 + \frac{2L(D + r)r}{\eps}\right)}\right\}.\end{equation}
	The fact that the third bound implies the first is clear, and the second is 
	implied by the assumption $2LD^2 > \eps$. 
	
	Our goal is to first show that if $g(u) > r$, then we have an $\eps$-approximate minimizer; otherwise, we construct a range $[\ell, u]$ which contains some $\lambda$ with $g(\lambda) = r$, and we apply the Lipschitz condition Lemma~\ref{lem:g-lip} to prove correctness of our binary search. Recall that for every $\lambda$, the guarantees of $\oracleBall$ imply that $\normm{\ztl-\tztl}\le\delta$, and moreover 
	\[\normm{\tztl-x^{*}} \le\normm{\tztl-\ytl}+\normm{\ytl-x^{*}}\le D+r\]
	by convexity. Thus, if it holds that $\normmd{\nabla f(\tilde{z}_{t_{u}})}\le 
	r/u+L\delta$ in Line 7, then
	\[
	f(\tilde{z}_{t_{u}})-f(x^{*})\le\innerProduct{\nabla f(\tilde{z}_{t_{u}})}{\tilde{z}_{t_{u}}-x^{*}}\le\normmd{\nabla f(\tilde{z}_{t_{u}})}\left(D+r\right)\le\eps,
	\]
	for our choice of $u=2(D+r)r/\eps$ and $\delta\le\eps/(2L(D+r))$ \eqref{eq:deltabounds}.
	On the other hand, if $\normmd{\nabla f(\tilde{z}_{t_{u}})}\ge r/u+L\delta$,
	by Lipschitzness of the gradient and the guarantee $\normm{\tztl - \ztl} \le \delta$, we have $g(u)=u\normmd{\nabla f(z_{t_{u}})}\ge r$. Moreover,
	for $\ell=r/L(D+r)$, by Lipschitzness of the gradient from $x^*$,
	\[
	g(\ell)=\ell\normmd{\nabla f(z_{t_{\ell}})} \le\frac{r}{L(D+r)}(L(D+r))\le r.
	\]
	By continuity, it is clear that for some value $\lambda\in[\ell,u]$,
	$g(\lambda)=r$; we note the assumption $2LD^{2}>\eps$ guarantees
	that $\ell<u$, so the search range is valid. Next, if for some
	value of $\lambda$, $\ztl=x^{*}$, as long as $\delta\le\sqrt{2\eps/L}$,
	we have by smoothness
	\[
	f(\tztl)-f(x^{*})=f(\tztl)-f(\ztl)\le\frac{L\delta^2}{2}\le\eps.
	\]
	Otherwise, $\ztl$ is on the boundary of the ball around
	$\ytl$, so that we have the desired
	\[\normm{\tztl-\ytl}\ge r-\delta \ge \frac{11r}{12}.\]
	Moreover, \eqref{eq:lamdef}	implies
	\begin{align*}
	\normm{\tztl-(\ytl-\lambda\mmd\nabla f(\tztl))} & \le(1+L\lambda)\delta+\normm{\ztl-(\ytl-\lambda\mmd\nabla f(\ztl))}\\
	& =(1+L\lambda)\delta+\normm{\left(\lambda-\frac{r}{\normmd{\nabla f(\ztl)}}\right)\mmd\nabla f(\ztl)}\\
	& =(1+L\lambda)\delta+\normm{\left(\frac{g(\lambda)-r}{\normmd{\nabla f(\ztl)}}\right)\mmd\nabla f(\ztl)}\\
	& \le(1+Lu)\delta+\left|g(\lambda)-r\right|.
	\end{align*}
	So, as long as $\delta\le r/(12(1+Lu))$ and $|g(\lambda)-r|\le r/4$, we	have the desired $\half$-MS oracle guarantee
	\begin{align*}
	\normm{\tztl-(\ytl-\lambda\mmd\nabla f(\tztl))} & \le\frac{r}{12}+\frac{r}{4}\le\half\left(r-\delta\right)\\
	& \le\half\normm{\tztl-\ytl}.
	\end{align*}
	Thus, the algorithm can terminate whenever we can guarantee $|g(\lambda) 
	- r| \le r/4$. We can certify the value of $g(\lambda)$ via 
	$\lambda\normmd{\nabla f(\tztl)}$ up to additive error $L\lambda\delta \le 
	r/12$, so that  $|\lambda\normmd{\nabla f(\tztl)} - r| \le r/6$ implies 
	$|g(\lambda) - r| \le r/4$. Finally, let $\lambda^*$ be any value in $[\ell, 
	u]$ where $g(\lambda^*) = r$. By Lemma~\ref{lem:g-lip}, 
	\[\left|\lambda - \lambda^*\right| \le \frac{r}{12(L(2D+r))} \implies \left|g(\lambda) - r\right| \le \frac{r}{12} \implies |\lambda\normmd{\nabla f(\tztl)} - r| \le r/6,\text{ i.e. search terminates.}\]
	In conclusion, we can bound the number of calls required by Algorithm~\ref{alg:line-search-oracle} in executions of Lines 16 and 20 to $\oracleBall$ by
	\[\log\left((u - \ell)\cdot\left(\frac{r}{12(L(2D+r))}\right)^{-1}\right) \le \log\left(\frac{4Dr}{\eps} \cdot \frac{36LD}{r}\right).\]
\end{proof}

\mainclaimaccel*
\begin{proof}
	More specifically, we will return the point encountered in 
	Algorithm~\ref{alg:msaccel_framework} with the smallest function value, 
	in the case Proposition~\ref{prop:line_search} ever guarantees a point is 
	an $\eps$-approximate minimizer. Note that 
	Lemma~\ref{lem:domain_bound} implies that in each run of 
	Algorithm~\ref{alg:line-search-oracle}, it suffices to set $D = 
	3\sqrt{2}R$, where we recall (in its context) $\sqrt{2p_0} = \sqrt{2}R$, via 
	the proof of Proposition~\ref{prop:ball_constrained_error_bound}. 
	Recalling $R > r$, this implies that setting
	\[\delta = \frac{r}{12(1 + \frac{2L(D + r)r}{\eps})} \ge \frac{r}{12 + 
	\frac{126LRr}{\eps}}\]
	suffices in the guarantees of $\oracleBall$. Moreover, if the assumption 
	$\eps < 2LD^2$ in Proposition~\ref{prop:line_search} does not hold, 
	smoothness implies we may return any $x_k$. The oracle complexity 
	follows by combining Proposition~\ref{prop:line_search} with 
	Proposition~\ref{prop:ball_constrained_error_bound}.
\end{proof}

\section{Trust region subproblems}\label{app:trust-region}

We give the algorithm for solving the trust region subproblem below.
\begin{algorithm}
	\begin{algorithmic}[1]
		\caption{$\textsc{SolveTR}(\bx, r, g, \mh, \mm, 
		\Delta)$}\label{alg:tr-method}
		\STATE Let $0 < \mu \le L$ so $\mu \mm \preceq \mh \preceq 
		L\mm$, and let $\Delta > 0$. 
		\STATE $\hg \gets g - \mh \bx$
		\STATE $\ell \gets 0$, $u \gets \frac{\normmd{\hg}}{r}$, $\iota \gets 
		\frac{\Delta\mu^2}{\normmd{\hg}}$
		\IF{$\normms{\mh^\dagger \hg} \le r$}
		\RETURN $\mh^\dagger \hg$
		\ELSE
		\STATE $\lam \gets \frac{\ell+u}{2}$, $\lam^- \gets \lam - \iota$
		\WHILE{\NOT $(\normms{\left(\mh + \lam\mm\right)^{\dagger}\hg} 
		\le r)$ \AND $(r < \normms{\left(\mh + 
		\lam^-\mm\right)^{\dagger}\hg})$}\label{line-SolveTR-condition}
		\IF{$\normms{\left(\mh + \lam\mm\right)^{\dagger}\hg} \le r$}
		\STATE $u \gets \lam$, $\lam \gets \frac{\ell+u}{2}$, $\lam^- \gets 
		\lam-\iota$
		\ELSE
		\STATE $\ell \gets \lam$, $\lam \gets \frac{\ell+u}{2}$, $\lam^- 
		\gets \lam-\iota$
		\ENDIF
		\ENDWHILE
		\RETURN $\left(\mh+\lam \mm\right)^\dagger \hg$
		\ENDIF
	\end{algorithmic}
\end{algorithm}

For simplicity, we first focus on developing technical results for the trust region problem of the following form (below $\mathbf{0}$ is the origin)
\begin{equation}
\min_{x\in\ballsett{\mathbf{0}}}Q(x)\defeq-g^{\top}x+\half x^{\top}\mh x;\label{eq:qcqozero}
\end{equation}
our final guarantees will be obtained by an appropriate linear shift. All results in this section assume $\mu\mm \preceq \mh\preceq L\mm$ for some $0 < \mu \le L$, which in particular implies that $\mh$ and $\mm$ share a kernel. We first state a helpful monotonicity property which will be used throughout.
\begin{lem}
\label{lem:lammonotone}
	$\normm{(\mh + \lambda\mm)^\dagger g}$ is monotonically decreasing in $\lambda$, for any vector $g$.
\end{lem}
\begin{proof}
We will refer to the projection onto the column space of $\mm$, i.e. $\mm\mmd$, by $\tilde{\mI}$. To show the lemma, it suffices to prove that
\[(\mh + \lambda \mm)^\dagger\mm(\mh + \lambda \mm)^\dagger\]
is monotone in the Loewner order. Denoting $\tmh \defeq \mmdh \mh \mmdh$,
\begin{equation}\label{eq:inversehlm}\left(\mh + \lambda \mm\right)^\dagger = \left(\mmh\left(\tmh + \lambda \tmi\right)\mmh\right)^\dagger = \mmdh\left(\tmh + \lambda \tmi\right)^\dagger\mmdh.\end{equation}
Therefore, it suffices to show that
\[\mmdh\left(\tmh + \lambda \tmi\right)^\dagger\left(\tmh + \lambda \tmi\right)^\dagger\mmdh\]
is monotone in the Lowener order, which follows as $\tmh$ and $\tmi$ commute.
\end{proof}

Next, we characterize the minimizer to \eqref{eq:qcqozero}.
\begin{lem}
	\label{lem:quadchar}A solution to
	(\ref{eq:qcqozero}) is given by $\quadmin=(\mh+\lambda\mm)^\dagger g$
	for a unique value of $\lambda\ge0$. Unless $\lambda=0$, $\normm{\quadmin} =r$.
\end{lem}

\begin{proof}
	By considering the optimality conditions of the Lagrange dual problem
	\[
	\min_{x}\max_{\lambda\ge0}-g^{\top}x+\half x^{\top}\mh 
	x+\frac{\lambda}{2}\left(x^\top \mm x-r^{2}\right),
	\]
	either $\lambda=0$ and the minimizer $\mh^\dagger g$ is in $\ballsett{0}$,
	or there is $\quadmin=(\mh+\lambda\mm)^\dagger g$ on the region boundary (linear shifts in the kernel of $\mm$ do not affect the $\mm$ norm constraint or the objective, so we may restrict to the column space without loss of generality). Uniqueness of $\lambda$ then follows from Lemma~\ref{lem:lammonotone}.
\end{proof}

Next, we bound how tightly we must approximate the value $\lambda$
in order to obtain an approximate minimizer to (\ref{eq:qcqo}).
\begin{lem}
	\label{lem:lipquad}Suppose $g \in \textup{Im}(\mm)$, and $\normm{\mh^{\dagger}g}>r$.	Then, for $\lambda^{*}>0$ such that $\normm{(\mh+\lambda^{*}\mm)^{\dagger}g}=r$,
	and any $\lambda > 0$ such that 
	$|\lambda-\lambda^{*}|\le\frac{\Delta\mu^{2}}{\normmd{g}}$,
	we have 
	\begin{equation}
	\normm{\left(\mh+\lambda\mm\right)^{\dagger}g-\quadmin}\le\Delta.\label{eq:lipquad}
	\end{equation}
\end{lem}
\begin{proof}
	We follow the notation of Lemma~\ref{lem:lammonotone}. Recalling \eqref{eq:inversehlm}, we expand
	\begin{equation}\label{eq:mess}\normm{\left(\mh+\lambda\mm\right)^{\dagger}g-\quadmin}^2 = \tg^\top \left(\left(\tmh + \lambda\tmi\right)^{\dagger} - \left(\tmh + \lambda^*\tmi\right)^{\dagger}\right)^2 \tg.\end{equation}
	Here, we defined $\tg = \mmdh g$. Note that $\norm{\tg}_2^2 = \normmd{g}^2$, where we used $g \in \textup{Im}(\mm)$. Without loss of generality, since $\tmh+\lambda\tmi$
	commute for all $\lam$ therefore simultaneously diagonalizable, suppose we are in the basis where $\tmh$ is diagonal and
	has diagonal entries $\left\{ h_{i}\right\} _{i\in[d]}$. Expanding the right hand side of (\ref{eq:mess}), we have
	\begin{align*}
	\sum_{i\in[d]}\tg_{i}^2\left(\frac{1}{h_{i}+\lambda}-\frac{1}{h_{i}+\lambda^{*}}\right)^{2}
	 & 
	=\sum_{i\in[d]}\tg_{i}^{2}\left(\frac{(\lambda^{*}-\lambda)^{2}}{(h_{i}+\lambda)^{2}(h_{i}+\lambda^{*})^{2}}\right)\\
	& 
	\le\sum_{i\in[d]}\frac{\tg_{i}^{2}}{\mu^{4}}\left(\frac{\Delta\mu^{2}}{\norm 
	\tg_{\mm^\dagger}}\right)^{2}\le\Delta^{2}.
	\end{align*}
	In the last inequality, note that whenever $h_i \neq 0$, it is at least $\mu$ by strong convexity in $\|\cdot\|_{\mm}$, and whenever $h_i$ is zero, so is $\tg_i$, by the assumption on $g$ and the fact that $\mm$ and $\mh$ share a kernel.
\end{proof}
Finally, by combining these building blocks, we obtain a procedure
for solving (\ref{eq:qcqo}) to high accuracy.

\ghsolve*
\begin{proof}
	First, for $\hg \defeq g - \mh \bx$, we have the equivalent problem
	\[
	\argmin_{\normm{y}\le r}-g^{\top}(y+\bx)+\half(y+\bx)^{\top}\mh(y+\bx)=\argmin_{y \in \ballsett{0}}-\hg^{\top}y+\half y^{\top}\mh y.
	\]
	Following Lemma \ref{lem:quadchar}, in Line 4 we verify whether for the optimal solution, $\lambda=0$, using one linear system solve. If not, by monotonicity of $\normm{(\mh+\lambda\mm)^{\dagger}\hg}$ in $\lambda$ (Lemma~\ref{lem:lammonotone}),
	it is clear that the value $\lambda^{*}$ corresponding to the solution lies in the range $[\ell, u] = [0,\normmd{\hg}/r]$, by
	\[\normm{\left(\mh + \frac{\normmd{\hg}}{r}\mm\right)^{\dagger}\hg}^2 \le r^2.\] 
	This follows from e.g. the characterization \eqref{eq:inversehlm}. Therefore, Lemma \ref{lem:lipquad}
	shows that it suffices to perform a binary search over this region
	to find a value $\lambda$ with additive error $\iota=\frac{\Delta\mu^{2}}{\normmd{\hat{g}}}$
	to output a solution $\tx$ of the desired accuracy. We note that we may check feasibility in $ \ballsett{0}$ by computing the value of $\normm{(\mh+\lambda\mm)^\dagger \hg}$ due to Lemma~\ref{lem:quadchar}, and it suffices to output the
	larger value of $\lambda$ amongst the endpoint of the interval of length $\iota$ containing $\lambda^{*}$, reflecting our termination condition in Line~\ref{line-SolveTR-condition}.
\end{proof} %
\section{Accelerated Newton method}
\label{app:newton-appendix}
This section gives the guarantees of Algorithm~\ref{alg:nesterov_h}, and in particular a proof of Theorem~\ref{thm:ballopt}. Throughout, assume $\mu\mm \preceq \mh \preceq L\mm$, where $\mh = \nabla^2 f(\bx)$. We note that Line~\ref{line:AGD-solve-sub} of Algorithm~\ref{alg:nesterov_h} is approximately implementing the step
\begin{equation}\label{eq:idealstep}
\begin{aligned}z_{k + 1}^{\textup{ideal}} &\gets \argmin_{z\in\ballset}\left\{\innerProduct{\nabla f(y_k)}{z} + \frac{1-\alpha}{2}\norm{z - z_k}_{\mh}^2 + \frac{\alpha}{2}\norm{z - y_k}_{\mh}^2\right\}\\
&=\argmin_{z\in\ballset}\left\{\innerProduct{\nabla f(y_k) - (1 - \alpha)\mh z_k - \alpha \mh y_k}{z} +\half z^\top \mh z\right\},\end{aligned}\end{equation}
with the guarantee $\normm{z_{k + 1}^{\textup{ideal}} - z_{k + 1}} \le \Delta$. Throughout, we denote $\ballmin$ as the minimizer of $f$ in $\ballset$.
\begin{lem}
	\label{lem:mirrdescexact}Consider a single iteration of Algorithm
	\ref{alg:nesterov_h} from a pair of points $x_{k},z_{k}$. We have
	\begin{align*}
	f(y_{k})+\innerProduct{\nabla 
	f(y_{k})}{z_{k + 1}^{\textup{ideal}}-y_{k}}+\frac{\alpha}{2}\norm{y_{k}-z_{k + 1}^{\textup{ideal}}}_{\mh}^{2}+\frac{1-\alpha}{2}\norm{z_{k}-z_{k + 1}^{\textup{ideal}}}_{\mh}^{2}\\
	\le 
	f(\ballmin)+\frac{1-\alpha}{2}\norm{z_{k}-\ballmin}_{\mh}^{2}-\half\norm{z_{k + 1}^{\textup{ideal}}-\ballmin}_{\mh}^{2}.
	\end{align*}
\end{lem}
\begin{proof}
	By the first-order optimality conditions of $z_{k+1}^{\textup{ideal}}$ with respect
	to $\ballmin$,  
	\begin{align*}
	\innerProduct{\nabla f(y_{k})}{z_{k + 1}^{\textup{ideal}}-\ballmin} & 
	\le\frac{1-\alpha}{2}\left(\norm{z_{k}-\ballmin}_{\mh}^{2}-\norm{z_{k}-z_{k + 1}^{\textup{ideal}}}_{\mh}^{2}\right)\\
	& 
	+\frac{\alpha}{2}\left(\norm{y_{k}-\ballmin}_{\mh}^{2}-\norm{y_{k}-z_{k+1}^{\textup{ideal}}}_{\mh}^{2}\right)-\half\norm{z_{k+1}^{\textup{ideal}}-\ballmin}_{\mh}^{2}.
	\end{align*}
	Here, we twice-used the well-known identity 
	$\innerProduct{\mh(z_{k+1}^{\textup{ideal}}-x)}{z_{k+1}^{\textup{ideal}}-\ballmin}=\half\norm{z_{k+1}^{\textup{ideal}}-\ballmin}_{\mh}^{2}+\half\norm{z_{k+1}^{\textup{ideal}}-x}_{\mh}^{2}-\half\norm{x-\ballmin}_{\mh}^{2}$.
	Rearranging this and using strong convexity, where we recall 
	$\alpha=c^{-1}$,
	\begin{align*}
	f(y_{k})+\innerProduct{\nabla 
	f(y_{k})}{z_{k+1}^{\textup{ideal}}-y_{k}}+\frac{\alpha}{2}\norm{y_{k}-z_{k+1}^{\textup{ideal}}}_{\mh}^{2}+\frac{1-\alpha}{2}\norm{z_{k}-z_{k+1}^{\textup{ideal}}}_{\mh}^{2}\\
	\le\left(f(y_{k})+\innerProduct{\nabla 
	f(y_{k})}{\ballmin-y_{k}}+\frac{\alpha}{2}\norm{y_{k}-\ballmin}_{\mh}^{2}\right)+\frac{1-\alpha}{2}\norm{z_{k}-\ballmin}_{\mh}^{2}-\frac{1}{2}\norm{z_{k+1}^{\textup{ideal}}-\ballmin}_{\mh}^{2}\\
	\le 
	f(\ballmin)+\frac{1-\alpha}{2}\norm{z_{k}-\ballmin}_{\mh}^{2}-\half\norm{z_{k+1}^{\textup{ideal}}-\ballmin}_{\mh}^{2}.
	\end{align*}
\end{proof}
Next, we modify the guarantee of Lemma \ref{lem:mirrdescexact} to
tolerate an inexact step on the point $z_{k+1}$. We use the following
lemma.
\begin{lem}
	\label{lem:approxfoo}Suppose the convex function $h$ is $L$-smooth
	in $\normm{\cdot}$ in a region $\xset$ with bounded $\normm{\cdot}$
	diameter $2r$, and $x_{h}$ is the minimizer of $h$ over $\xset$.
	Then for $\hx$ with
	\[
	\normm{\hx-x_{h}}\le\Delta,\;\normmd{\nabla h(\hx)}\le G,
	\]
	and $\nabla h(\hx), \nabla h(x_h) \in \textup{Im}(\mm)$, we have for all $x\in\xset$, $\innerProduct{\nabla 
	h(\hx)}{\hx-x}\le2L\Delta r+G\Delta$.
\end{lem}

\begin{proof}
	First-order optimality of $x_{h}$ against $x\in\xset$ implies
	$\innerProduct{\nabla h(x_{h})}{x_{h}-x}\le0$. The conclusion follows:
	\begin{align*}
	\innerProduct{\nabla h(\hx)}{\hx-x} & =\innerProduct{\nabla 
	h(x_h)}{x_h-x}+\innerProduct{\nabla h(\hx)-\nabla 
	h(x_h)}{x_h-x}+\innerProduct{\nabla h(\hx)}{\hx-x_h}\\
	& \le0+2L\Delta r+G\Delta.
	\end{align*}
\end{proof}
Putting together Lemma \ref{lem:mirrdescexact} and Lemma \ref{lem:approxfoo} we have the following corollary.
\begin{cor}
	\label{corr:mirrdescinexact}Consider a single iteration of Algorithm
	\ref{alg:nesterov_h} from a pair of points $x_{k},z_{k}$.
	Also, assume that $\normm{\bx-x^{*}}\le D$, where $x^{*}$ is
	the global optimizer of $f$. Then,
	\begin{align*}
	f(y_{k})+\innerProduct{\nabla 
	f(y_{k})}{z_{k+1}-y_{k}}+\frac{\alpha}{2}\norm{y_{k}-z_{k+1}}_{\mh}^{2}+\frac{1-\alpha}{2}\norm{z_{k}-z_{k+1}}_{\mh}^{2}\\
	\le 
	f(\ballmin)+\frac{1-\alpha}{2}\norm{z_{k}-\ballmin}_{\mh}^{2}-\half\norm{z_{k+1}-\ballmin}_{\mh}^{2}+L\Delta(5r + D).
	\end{align*}
\end{cor}

\begin{proof}
	First, the Hessian of the objective being minimized in \eqref{eq:idealstep} is $\mh$, so the objective is $L$-smooth w.r.t $\|\cdot\|_{\mm}$ over a region $\ballset$ of bounded diameter $2r$. From Lemma \ref{lem:mirrdescexact} and
	\ref{lem:approxfoo} we have that the first-order optimality condition
	is correct up to an additive $2L\Delta r+G\Delta$,
	where $G$ is a bound on the gradient norm of the objective at 
	$z_{k+1}$.
	The conclusion follows from $\mh \mmd \mh \preceq L^2\mm$ by smoothness, so that
	\begin{align*}
	G = \normmd{\nabla 
	f(y_{k})+(1-\alpha)\mh(z_{k+1}-z_{k})+\alpha\mh(z_{k+1}-y_{k})}
	& \le\normmd{\nabla f(y_{k})}+2(1-\alpha)Lr+2\alpha Lr\\
	& \le L(D+r)+2Lr.
	\end{align*}
	In the final inequality, we used 
	$\normm{y_{k}-x^{*}}\le\normm{\bx-x^{*}}+\normm{y_{k}-\bx}\le 
	D+r$,
	and Lipschitzness of $\nabla f$ w.r.t $\|\cdot\|_{\mm}$.
\end{proof}
With this in hand, we can quantify how much progress is made in each iteration of the algorithm.
\begin{lem}
	\label{lem:potfunc}Consider a single iteration of Algorithm 
	\ref{alg:nesterov_h}
	from a pair of points $x_{k},z_{k}$. Also, assume that 
	$\normm{\bx-x^{*}}\le D$,
	where $x^{*}$ is the global optimizer of $f$. Then,
	\[
	f(x_{k+1})-f(\ballmin)+\frac{1}{2c}\norm{z_{k+1}-\ballmin}_{\mh}^{2}\le\left(1-\frac{1}{c}\right)\left(f(x_{k})-f(\ballmin)+\frac{1}{2c}\norm{z_{k}-\ballmin}_{\mh}^{2}\right)+Lc^{-1}\Delta(5r+D).
	\]
\end{lem}

\begin{proof}
	By stability and 
	$x_{k+1}-y_{k}=\alpha(z_{k+1}-y_{k})+(1-\alpha)(x_{k}-y_{k})$
	from the definition of the algorithm,
	\begin{align*}
	f(x_{k+1}) & \le f(y_{k})+\innerProduct{\nabla 
	f(y_{k})}{x_{k+1}-y_{k}}+\frac{c}{2}\norm{x_{k+1}-y_{k}}_{\mh}^{2}\\
	& =(1-\alpha)\left(f(y_{k})+\innerProduct{\nabla 
	f(y_{k})}{x_{k}-y_{k}}\right)+\alpha\left(f(y_{k})+\innerProduct{\nabla 
	f(y_{k})}{z_{k+1}-y_{k}}\right)\\
	& 
	+\frac{c}{2}\norm{\alpha(z_{k+1}-y_{k})+(1-\alpha)(x_{k}-y_{k})}_{\mh}^{2}\\
	& \le(1-\alpha)f(x_{k})+\alpha\left(f(y_{k})+\innerProduct{\nabla 
	f(y_{k})}{z_{k+1}-y_{k}}+\frac{1}{2}\norm{z_{k+1}-(1-\alpha)z_{k}-\alpha
	 y_{k}}_{\mh}^{2}\right)\\
	& \le(1-\alpha)f(x_{k})+\alpha\left(f(y_{k})+\innerProduct{\nabla 
	f(y_{k})}{z_{k+1}-y_{k}}+\frac{\alpha}{2}\norm{y_{k}-z_{k+1}}_{\mh}^{2}+\frac{1-\alpha}{2}\norm{z_{k}-z_{k+1}}_{\mh}^{2}\right).
	\end{align*}
	The second inequality used convexity and $(1-\alpha)(\alpha 
	z_{k}+x_{k})=(1-\alpha^{2})y_{k}$,
	which implies
	\[
	\alpha(z_{k+1}-y_{k})+(1-\alpha)(x_{k}-y_{k})=\alpha(z_{k+1}-(1-\alpha)z_{k}-\alpha
	 y_{k}),
	\]
	and the third inequality used convexity of the norm squared. Substituting
	the earlier bound from Corollary \ref{corr:mirrdescinexact} yields
	the conclusion, recalling $\alpha=c^{-1}$.
\end{proof}
Now we are ready to prove the main result for the implementation of the ball optimization oracle, restated below.
\ballopt*
\begin{proof}
	First, for each
	iteration $k$, define the potential function
	\[
	\Phi_{k}\defeq 
	f(x_{k})-f(\ballmin)+\frac{1}{2c}\norm{z_{k}-\ballmin}_{\mh}^{2}.
	\]
	By applying Lemma \ref{lem:potfunc}, and defining $E\defeq 
	L\Delta(5r+D) = \frac{\mu\delta^2}{4c}$ by the definition of $\Delta$ in Algorithm~\ref{alg:nesterov_h},
	we have
	\[
	\Phi_{k+1}\le\left(1-\frac{1}{c}\right)\Phi_{k}+\frac{E}{c}.
	\]
	Telescoping this guarantee and bounding the resulting geometric series in $E$ yields
	\begin{equation}
	\Phi_{k}\le\left(1-\frac{1}{c}\right)^{k}\Phi_{0}+E.\label{eq:noisypot}
	\end{equation}
	Now, recalling $x_{0}=z_{0}=\bx$, we can bound the
	initial potential by
	\[
	\Phi_{0}\le\innerProduct{\nabla 
	f(\ballmin)}{\bx-\ballmin}+\frac{c}{2}\norm{\bx-\ballmin}_{\mh}^{2}+\frac{1}{2c}\norm{\bx-\ballmin}_{\mh}^{2}\le
	 LDr+Lcr^{2}.
	\]
	where we used $\mh\preceq L\mm$. Next, note that whenever we have $\Phi_{k}\le\mu\delta^{2}/2c$, we
	have
	\[
	\frac{1}{2c}\norm{z_{k}-\ballmin}_{\mh}^{2}\le\frac{\mu\delta^{2}}{2c}\Rightarrow\norm{z_{k}-\ballmin}_{\mm}\le\delta,
	\]
	where we used $\mh\succeq\mu\mm$. Thus, as $E=\mu\delta^{2}/4c$, running for
	\begin{equation}
	\label{eqn:num_iter}
	k=O\left(c\log\left(\frac{Lc(Dr+cr^{2})}{\mu\delta^{2}}\right)\right) = 
	O\left(c\log\left(\frac{\kappa (D+r)c}{\delta}\right)\right)
	\end{equation}
	iterations suffices to guarantee $\Phi_{k}\le\mu\delta^{2}/2c$ via
	(\ref{eq:noisypot}), and therefore implements a $(\delta,r)$-ball optimization oracle at $\bar{x}$. It remains to bound the complexity of each iteration. For this, we apply Proposition~\ref{prop:ghsolve} with the parameter
	$\Delta=\mu\delta^{2}/(4Lc(5r+D))$, and compute
	\[
	\normmd{\mh(\bx-(1-\alpha)z_{k}-\alpha y_{k})+\nabla f(y_{k})}\le 
	L\normm{\bx-(1-\alpha)z_{k}-\alpha y_{k}}+\normmd{\nabla 
	f(y_{k})}\le 2Lr+LD.
	\]
	Altogether, the number of linear system solves in the step is then bounded by 
	\[
	O\left(\log\left(\frac{L^{2}(D+r)^{2}}{\mu^{2}}\cdot\frac{Lc(D+r)}{\mu\delta^{2}r}\right)\right)
	 = O\left(\log\left(\frac{\kappa (D+r)c}{\delta}\right)\right),
	\]
	where the first term is due to the squared norm and $\mu^{-2}$, and
	the second is due to $(r\Delta)^{-1}$, in the bound of Proposition~\ref{prop:ghsolve}. The final bound follows from the assumption $\delta < r$. Combining with \eqref{eqn:num_iter} yields the claim.
\end{proof}
\section{Proof of Lemma~\ref{lem:quasi-to-stable}}
\label{app:qsc}
Here, we prove Lemma~\ref{lem:quasi-to-stable}, which shows quasi-self-concordance implies Hessian stability.

\quasisc*

\begin{proof}
	Let $x,y\in\R^{d}$ be arbitrary and let $x_{t}\defeq x+t(y-x)$ for
	all $t\in[0,1]$. Then for all $u\in\R^{d}$,
	\[
	\frac{d}{dt}\left(\norm u_{\nabla^{2}f(x_{t})}^{2}\right)=\frac{d}{dt}\left(u^{\top}\hess f(x_{t})u\right)=\nabla^{3}f(x_{t})[u,u,y-x].
	\]
	The result follows from
	\begin{align*}
	\left|\log\left(\norm u_{\nabla^{2}f(y)}^{2}\right)-\log\left(\norm u_{\nabla^{2}f(x)}^{2}\right)\right| & =\left|\int_{0}^{1}\frac{\nabla^{3}f(x_{t})[u,u,y-x]}{\norm u_{\hess f(x_{t})}^{2}}dt\right|\leq M\norm{y-x}.
	\end{align*}
\end{proof}

\section{Proofs for applications}
\label{app:application}

\citethis*
\begin{proof}
	Let the minimizer of $\tilde{f}(x)$ be $\tilde{x}$: observe by 
	Lemma~\ref{lem:dist_falls} that $\norm{x_0 - \tilde{x}}_\mm \leq 
	\norm{x_0 - x^*}_\mm \leq R$. Note that $\tilde{f}(x)$ is $L + 
	\eps/55R^2$-smooth and $\eps/55R^2$-strongly convex in 
	$\normm{\cdot}$, and since the iterates of 
	Algorithm~\ref{alg:msaccel_framework} never are more than $D = 
	3\sqrt{2}R$ away from $\tilde{x}$ (Lemma~\ref{lem:domain_bound}), by 
	the triangle inequality and $(1 + 3\sqrt{2})^2 \le 55/2$, $\tilde{f}$ 
	approximates $f$ to an additive error $\eps/2$ for all iterates. Next, 
	letting $r = 1/M$, it follows from Lemma~\ref{lem:quasi-to-stable} that 
	$g$ is $(r, e)$-Hessian stable in $\ell_2$, so that $f$ is $(r, e)$-Hessian 
	stable in $\normm{\cdot}$ (see Lemma~\ref{lem:stillstable}). It follows 
	from the definition of Hessian stability that $\tilde{f}$ is also $(r, 
	e)$-Hessian stable in $\normm{\cdot}$ (see  
	Lemma~\ref{lem:morestable}). Finally, the conclusion follows from 
	combining the guarantees of Theorem~\ref{thm:mainclaim-accel} and 
	Theorem~\ref{thm:ballopt}, where it suffices to minimize $\tilde{f}$ to 
	$\eps/2$ additive error.
\end{proof}

\begin{lem}
\label{lem:stillstable}
Let $g: \R^n \rightarrow \R$ be $M$-QSC in $\ell_2$. Then, $f(x) = g(\ma x)$ is $M$-QSC in $\norm{\cdot}_{\ma^\top\ma}$, for $\ma \in \R^{n\times d}$, $f: \R^d \rightarrow \R$. 
\end{lem}
\begin{proof}
Recall the condition on $g$ implies for all $u, h, x \in \R^d$,
\[
\left|\nabla^{3}g(\ma x)[\ma u,\ma u,\ma h]\right|\le M\|\ma h\|_2 \|\ma u\|_{\nabla^{2}g(y)}^{2}.
\]
Using this, and recalling $\norm{\ma h}_2 = \norm{h}_{\ma^\top\ma}$, $\nabla^2 f(x) = \ma^\top \nabla^2 g(\ma x) \ma$, the result follows:
\[\left|\nabla^{3}f(x)[u,u,h]\right|=\left|\nabla^{3}g(\ma x)[\ma u,\ma u,\ma h]\right| \le M\norm{ h}_{\ma^\top\ma}\norm{u}_{\nabla^2 f(x)}^{2}.\]
\end{proof}
\begin{lem}
\label{lem:morestable}
Suppose $f$ is $(r, c)$-stable in $\norm{\cdot}$. Then for any matrix $\mm$ and $\lam \ge 0$, $\tilde{f}$ defined by $\tilde{f}(x) = f(x) + \frac{\lam}{2}x^\top\mm x$ is also $(r, c)$-stable in $\norm{\cdot}$.
\end{lem}
\begin{proof}
It suffices to show that for $x, y \in \R^d$ with $\norm{x - y} \le r$,
\[c^{-1}\nabla^2 \tilde{f}(y) \preceq \nabla^2 \tilde{f}(x) \preceq c\nabla^2 \tilde{f}(y).\]
This immediately follows from $\nabla^2\tilde{f}(x) = \nabla^2 f(x) + \lam \mm$, and combining
\begin{align*}
c^{-1}\nabla^2f(y) \preceq \nabla^2f(x) \preceq c\nabla^2f(y),\; c^{-1}\lam\mm \preceq \lam\mm \preceq c\lam\mm.
\end{align*}
\end{proof}

\begin{lem}
Let $f$ be a convex function with minimizer $x^*$, and let $\mm$ be a positive semidefinite matrix. If $f_t(x) = f(x) + \frac{t}{2} \norm{x - y}_\mm^2$ is minimized at $x_t$, then for all $u \geq 0$, $\normm{x_u - y} \leq \normm{x^* - y}$. 
\label{lem:dist_falls}
\end{lem}
\begin{proof}
By the KKT conditions for $f_t$ we observe
\[
\grad f(x_t) = - t \mm (x_t - y). 
\]
Taking derivatives of this with respect to $t$ we obtain
\[
\grad^2 f(x_t) \frac{d x_t}{dt} = - \mm (x_t - y) - t \mm \frac{d x_t}{dt}
\]
or
\[
\frac{d x_t}{dt} = - \left(\grad^2 f(x_t) + t \mm \right)^\dagger \mm (x_t - y).
\]
Now we have
\begin{align*}
\norm{x_u - y}_{\mm}^2 - \norm{x^* - y}_{\mm}^2 &= 2\int_{0}^u \left(\frac{d}{dt} x_t\right)^\top \mm (x_t - y) dt \\
&= - 2\int_{0}^{u} \norm{x_t - y}_{\mm (\grad^2 f(x_t) + t \mm)^\dagger \mm}^2 dt \leq 0
\end{align*}
as desired. 
\end{proof}

\begin{lem}[Approximation of $\lse_t$]
\label{lem:lseapprox}
For all $y \in \R^n$, 
\[|\lse_t(y) - \max_{i \in [n]} y_i| < t\log n.\]
\end{lem}
\begin{proof}
This follows from the facts that for $z \in \Delta^n$ the probability simplex, the entropy function $h(z) \defeq \sum_{i \in [n]} z_i\log z_i$ has range $[-\log n, 0]$, $\max_{i \in [n]} y_i = \max_{z \in \Delta^n} z^\top y$, and by computation
\[\lse_t(y) = \max_{z \in \Delta^n} z^\top y - th(z).\]
\end{proof}
\subsection{Softmax calculus}
\begin{proof}[Proof of Lemma~\ref{lem:softmax_stable}]
We will prove $1$-smoothness and $2$-QSC for $\lse$, which implies the claims by chain rule. Let $S\defeq\sum_{i\in[n]}\exp(x_{i})$, and let $g\in\R^{n}$ with
$g_{i}=\exp(x_{i})/S$, $\mg\defeq\mdiag(g)$. Direct calculation
reveals that for all $i,j,k\in[n]$
\begin{align*}
\frac{\partial}{\partial x_{i}}\lse(x) & =g_{i},\\
\frac{\partial^{2}}{\partial x_{i}\partial x_{j}}\lse(x) & =\indicVec{i=j}g_{i}-g_{i}g_{j}\text{, and}\\
\frac{\partial^{3}}{\partial x_{i}\partial x_{j}\partial x_{k}}\lse(x) & =\indicVec{i=j=k}g_{i}-\indicVec{i=j}g_{i}g_{j}-\indicVec{i=k}g_{i}g_{k}-\indicVec{j=k}g_{j}g_{k}+2g_{i}g_{j}g_{k}.
\end{align*}
Therefore, we have that $\hess \lse(x)=\mg-gg^{\top}$. Now, note that
$g_{i}\geq0$ for all $i\in[n]$ and $\norm g_{1}=1$. By Cauchy-Schwarz,
\[
\left(g^{\top}u\right)^{2}=\left(\sum_{i\in[n]}g_{i}u_{i}\right)^{2}\leq\left(\sum_{i\in[n]}g_{i}\right)\left(\sum_{i\in[n]}g_{i}u_{i}^{2}\right)=u^{\top}\mg u\leq\norm g_{1}\norm u_{\infty}^{2}=\norm u_{\infty}^{2}\,.
\]
This implies that $0\preceq\nabla^{2}\lse(x)\preceq\mg$, and the first part
follows. Further, letting $\mh\defeq\mdiag(h)$ and $\mU\defeq\mdiag(u)$
we have from direct calculation
\begin{align*}
u^{\top}\mg\mU h-(g^{\top}u)(h^{\top}\mg u) & =u^{\top}\nabla^{2}\lse(x)\mU h,\\
-(g^{\top}u)(h^{\top}\mg u)+(g^{\top}u)^{2}(g^{\top}h) & =-(g^{\top}u)\left(u^{\top}\nabla^{2}\lse(x)h\right)\\
-(u^{\top}\mg u)(g^{\top}h)+(g^{\top}u)^{2}(g^{\top}h) & =-\left(u^{\top}\nabla^{2}\lse(x)u\right)(g^{\top}h).
\end{align*}
Combining these equations and the previous derivation of $\nabla^{3}f(x)$,
\begin{align}
\left|\nabla^{3}\lse(x)[u,u,h]\right| & =\left|u^{\top}\mg\mh u-2(g^{\top}u)(h^{\top}\mg u)-(u^{\top}\mg u)(g^{\top}h)+2(g^{\top}u)^{2}(g^{\top}h)\right|\nonumber \\
 & =\left|u^{\top}\hess \lse(x)\mU h-(g^{\top}u)\left(u^{\top}\nabla^{2}\lse(x)h\right)-\left(u^{\top}\hess \lse(x)u\right)(g^{\top}h)\right|\label{eq:bound0smax}\\
 & \leq\left|u^{\top}\hess \lse(x)\left(\mU h-(g^{\top}u)h\right)\right|+\left|g^{\top}h\right|\norm u_{\hess \lse(x)}^{2}.\nonumber 
\end{align}
Now, since $\hess \lse(x)\succeq 0$ we have 
\begin{equation}
\left|u^{\top}\hess \lse(x)\left(\mU h-(g^{\top}u)h\right)\right|\leq\norm u_{\hess \lse(x)}\norm{\left(\mU-(g^{\top}u)\mI\right)h}_{\hess \lse(x)}.\label{eq:bound1smax}
\end{equation}
Further, recall $\hess \lse(x)\preceq\mg$ and consequently 
\begin{align}
\norm{\left(\mU-(g^{\top}u)\mI\right)h}_{\hess \lse(x)}^{2} & \leq\norm{\left(\mU-(g^{\top}u)\mI\right)h}_{\mg}^{2}=\sum_{i\in[n]}h_{i}^{2}g_{i}\left(u_{i}-g^{\top}u\right)^{2}\nonumber \\
 & \leq\norm h_{\infty}^{2}\sum_{i\in[n]}g_{i}\left(u_{i}^{2}-2u_{i}(g^{\top}u)+(g^{\top}u)^{2}\right)\label{eq:bound2smax}\\
 & =\norm h_{\infty}^{2}\left(\left(\sum_{i\in[n]}g_{i}u_{i}^{2}\right)-2(g^{\top}u)^{2}+\left(g^{\top}u\right)^{2}\norm g_{1}\right)\\
 & =\norm h_{\infty}^{2}\norm u_{\hess \lse(x)}^{2}.\nonumber 
\end{align}
Combining (\ref{eq:bound0smax}), (\ref{eq:bound1smax}), (\ref{eq:bound2smax}),
and using $|g^{\top}h|\leq\norm g_{1}\norm h_{\infty}\leq\norm h_{\infty}$ and
$\norm h_{\infty}\le\norm h_{2}$, the result follows.
\end{proof}

\subsection{Proofs for $\ell_p$ regression}
\label{ssec:lpproofs}
\lpqsc*
\begin{proof}%
Let $\tilde{g}(x) = g(x) + {\mu} \norm{x-y}_2^2$. 
We observe that
\[
|\grad^3 \tilde{g}(x)[h, u, u]| = p(p-1)(p-2)\sum_{i = 1}^n h_i u_i^2 \abs{ 
x_i - b_i}^{p-3}
\]
and 
\[
\grad^2 \tilde{g}(x)[u, u] = \sum_{i = 1}^n u_i^2 \left( p(p-1) \lvert x - 
b\rvert_i^{p-2} + 2\mu \right). 
\]
Now, 
\begin{align*}
|\grad^3 \tilde{g}(x)[h, u, u]| &\leq p(p-1)(p-2) \norm{h}_\infty \sum_{i = 
1}^n u_i^2 \abs{ x_i - b_i}^{p-3} \\
&\leq \norm{h}_2 \sum_{i = 1}^n (p(p-1)(p-2))^{\frac{p}{3p-6}} u_i^2 \left( 
(p(p-1)(p-2))^{\frac{2}{3}} \abs{ x_i - b_i}^{p-2} \right)^{\frac{p-3}{p-2}} \\
&\leq O(p \mu^{-1/(p-2)}) \norm{h}_2  \sum_{i = 1}^n u_i^2 \left( p(p-1) 
\abs{ x_i - b_i}^{p-2} \right)^{\frac{p-3}{p-2}} \mu^{\frac{1}{p-2}} \\
&\leq O(p \mu^{-1/(p-2)}) \norm{h}_2  \sum_{i = 1}^n u_i^2 \left( p(p-1) 
\abs{ x_i - b_i}^{p-2} + 2\mu \right) \\
&\leq O(p \mu^{-1/(p-2)}) \norm{h}_2 \grad^2 \tilde{g}(x)[u, u]
\end{align*}
where we used that $u_i^2 \abs{ x_i - b_i}^{p-3}$ is nonnegative in the 
first line and $\norm{\cdot}_\infty \leq \norm{\cdot}_2$ in the second. In 
the third line we used that $(p(p-1)(p-2))^{\frac{p}{3p-6}} \leq 
p^\frac{3p}{3p-6} = p p^{\frac{6}{3p-6}} = O(p)$ since $p^{\frac{6}{3p-6}}$ 
is at most $9$ if $p \geq 3$. Finally in the fourth line we applied the 
inequality $x^\alpha y^{1-\alpha} \leq \max(x, y) \le x + y$ for 
nonnegative $x, y$, and $\alpha \in [0,1]$. The claim follows.  
\end{proof}

To prove~\Cref{lem:lp_dist} we use the following lemma from 
\cite{AdilKPS19}, with notation modified 
to our setting.\footnote{The function $\gamma_p(\lvert y \rvert, \Delta)$ 
	in their setting is at least $\norm{y}_p^p$.} 

\begin{lem}[{\citet[Lemma 
4.5]{AdilKPS19}}]\label{lem:lp-uniform-convexity}
	Let $p \in (1, \infty)$. Then for any two vectors $y,\Delta\R^n$,
	\[
	\norm{y}_p^p + v^\top \Delta + \frac{p-1}{p \cdot 2^p} 
	\norm{\Delta}_p^p \leq \norm{y + \Delta}_p^p
	\]
	where $v_i = p \lvert y_i \rvert^{p-2} y_i$ is the gradient of 
	$\norm{y}_p^p$.
\end{lem}

\lpdist*
\begin{proof}%
 Substituting $y = \ma x^* - b$, $\Delta = \ma(x - x^*)$ in 
 \Cref{lem:lp-uniform-convexity}, and simplifying gives
\[
\norm{\ma x^* - b}_p^p + \grad f(x^*)^\top (x - x^*) + \frac{p-1}{p 2^p} \norm{\ma (x - x^*)}_p^p \leq \norm{\ma x - b}_p^p.
\]
As $\grad f(x^*)^\top (x - x^*) = 0$ by optimality of $x^*$, we obtain 
\[
\frac{p-1}{p 2^p} \norm{\ma (x - x^*)}_p^p \leq  \norm{\ma x - b}_p^p - 
\norm{\ma x^* - b}_p^p = f -f^* \le \epsilon. %
\]
Now using $\norm{\cdot}_2 \leq n^{\frac{p-2}{2p}} \norm{\cdot}_p$ this implies 
\[
\norm{x - x^*}_{\mm}^p \leq 2^{p+1} n^{\frac{p-2}{2}} \epsilon.
\]
as $\frac{p}{p-1} \leq 2$ for $p \geq 3$. 
\end{proof}

\lpprogress*
\begin{proof}%
	We apply Algorithm~\ref{alg:ms-bacon} to 
	compute an $\epsilon_k=2^{-p}\epsilon_{k-1}$ approximate minimizer 
	of $f$ in 
	\begin{equation*}
	O\left((RM)^{2/3} \log^3 \left(\frac{LR^2}{\epsilon_{k}}(1+MR)\right) \log 
	\left(
	\frac{\epsilon_{k-1}}{\epsilon_{k}}\right)\right)
	\end{equation*}
	linear system solutions, with parameters $R,M$ and $L$ that 
	we bound as follows.
	
	By Lemma~\ref{lem:lp_dist}, 
	\[
	\norm{x_{k-1} 
	- x^*}_\mm^p \leq  2^{p+1} n^{\frac{p-2}{2}} 
	\epsilon_{k-1}  \defeq R^p
	\]
	 We add $\frac{\epsilon_{k}}{55 R^2} 
	\norm{x 
	- 
	x_k}_\mm^2$ to $f$ in obtaining $\tilde{f}$, and observe that the proof 
	of Corollary \ref{corr:citethis} only requires us to show that $\tilde{f}$ is 
	QSC. 
	By Lemma~\ref{lem:lp_qsc}, we see that $\tilde{f}$ is $M = O\left( p 
	\left(R^2/\epsilon_{k}\right)^{1/(p-2)}\right)$-QSC.  Therefore,  
	for any 
	$p \geq 3$ we have 
	\[
	RM = O\left(p R^{\frac{p}{p-2}} \epsilon_{k}^{-\frac{1}{p-2}}\right) = 
	O\Bigg(p \sqrt{n} 
	\bigg(\frac{2^{p+1}\epsilon_{k-1}}{\epsilon_k}\bigg)^{\frac{1}{p-2}}\Bigg)
	 =
	O(p \sqrt{n}),
	\]
	so the polynomial term in the running time is at most 
	$O\left((RM)^{2/3}\right) = O\left(p^{2/3} n^{1/3}\right)$. We now 
	bound the logarithmic factors in the runtime. Observe that for any $x$ 
	output by our MS oracle implementation we have that $\norm{x - 
	x^*}_\mm \leq 2 \sqrt{3} 
	R$ (\Cref{lem:domain_bound} with $\sigma=\half$). As the Hessian of 
	$f$ is $\ma^\top \md \ma$ where $\md_{ii} = 
	p(p-1) \lvert \ma x - b \rvert_i^{p-2}$ we may upper bound the 
	smoothness of $f$ (w.r.t. $\norm{\cdot}_{\mm}$) at all points 
	encountered during the algorithm by 
	\[
	O \left( p^2 \max_{\norm{x - x^*}_\mm \leq 2 \sqrt{3} R} \norm{\ma x - 
	b}_\infty^{p-2} \right) .
	\] 
	For any $x$ such that $\norm{x - x^*}_\mm \leq 2 \sqrt{3} R$ we 
	have  $\norm{\ma x - b}_\infty \leq \norm{\ma x^* - b}_\infty + 
	\norm{\ma (x - x^*)}_\infty \leq  \norm{\ma x^* - b}_p + \norm{x - 
	x^*}_\mm \leq (f^*)^{1/p} + 2 \sqrt{3} R$. Using the assumption $f^* \le 
	\epsilon_{k-1}/\delta \le \delta^{-1}n^{-\frac{p-2}{2}} R^p$,  we may 
	upper bound $L$ as
	\[
	L = O\left(4^p p^2 
	\left(1+\delta^{-\frac{1}{p}}n^{-\frac{p-2}{2p}}\right)^{p-2} 
	R^{p-2}\right).
	\] 
	Recalling that $R^p = 2^{2p+1} n^{\frac{p-2}{2}} \epsilon_k$, we obtain
	\begin{flalign*}
	\frac{LR^2}{\epsilon_{k}}(1+MR) &= O\prn*{4^p p^2 
	\left(1+\delta^{-\frac{1}{p}}n^{-\frac{p-2}{2p}}\right)^{p-2} 
	\frac{R^p}{\epsilon_k}(1+p\sqrt{n})} \\ &
	=
	O\prn*{16^p \left(\sqrt{n}+(n/\delta)^{1/p} \right)^{p-2} p^3\sqrt{n}}
	=
	O\prn*{17^p \left(\sqrt{n}+(n/\delta)^{1/p} \right)^{p}}
	\end{flalign*}
	Taking a logarithm yields
	\[
	\log \left( \frac{LR^2}{\epsilon}(1+MR) \right)  \leq O\left(p\log n + \log 
	\frac{n}{\delta}\right) = O\left(p \log 
	\frac{n}{\delta}\right). 
	\]
	Finally since $\log(\epsilon_{k-1}/\epsilon_{k}) = p$, combining the 
	above 
	bounds with the running time of Corollary~\ref{corr:citethis} gives a 
	bound of $O(p^{14/3} n^{1/3} \log^3(n/\delta))$ linear system solves as 
	desired. 
\end{proof}

\section{Lower bound proofs}\label{sec:lb-appendix}

Throughout, recall the coordinate progress notation
\[
\prog(x) \defeq \min \{ i\in[d] \mid |x_j|  \le r~\mbox{for all}~j\ge i \}.
\]

\subsection{Proof of \Cref{lem:robust-zero-chain-prog-bound}}

\proglem*

Before giving the proof, we remark that a number of 
papers~\citep{woodworth2017lower,carmon2019lower_i,
	diakonikolas2018lower,bubeck2019complexity} contain proofs for 
	variations of this claim featuring some differences between the types of 
	oracles considered, which do not materially affect the argument. The 
	proofs in these papers are distinct, and vary in the dimensionality they 
	require. Our argument below uses a random orthogonal transformation 
	similarly to~\citet{woodworth2017lower,carmon2019lower_i}, but uses a 
	more careful union bound \eqref{eq:prog-ub-1}, similarly to 
	that of~\citet{diakonikolas2018lower}, which allows for a much shorter proof 
	and also obtains tighter dimension bounds as in 
	\citet{diakonikolas2018lower,bubeck2019complexity}.

\begin{proof}
	Let $u_1,\ldots,u_d$ be the columns of $\mU$. \cref{def:robust-chain} 
	directly suggests an $r$-local oracle for $f_\mU(x) 
	= f(\mU^\top x)$: at query point $\bx$ the oracle 
	returns $f_\mU^{\bx}:\R^d \to \R$ such that
	\begin{equation*}
	f_\mU^{\bx}(x) = f(\inner{u_1}{x}, \ldots, \inner{u_{\prog(\mU^\top 
	\bx)}}{x}, 
	0, 
	\ldots, 0).
	\end{equation*}
	The $r$-robust zero-chain definition implies that $\oracle(\bx) 
	= f_\mU^{\bx}$ is a valid response for an 
	$r$-local oracle for $f_\mU$. Moreover, the oracle answer to query 
	$x_i$ 
	only depends on the first $\prog(\mU^\top x_i)$ columns of $u$. Define
	\begin{equation*}
	\mprog_i \defeq \max_{j\le i} \prog\left(\mU^\top x_j\right)
	\end{equation*}
	to be the highest progress attained up to query $i$. With this notation, 
	we wish to show that 
	\[
	\P\prn[\Bigg]{ \bigcap_{i\le N}\{\mprog_i \le i\}} \ge 1-\delta. 
	\]
	Note that at round $i+1$ the algorithm could query 
	$x_{i+1} = R \cdot u_{\mprog_i}$ which would satisfy 
	$\mprog_{i+1}=\prog(\mU^\top 
	x_{i+1}) = 1+\mprog_i$. Therefore, it is possible to choose queries so 
	that $\prog(\mU^\top x_i) = \mprog_i = i$. However, any faster increase 
	in 
	$\mprog_i$ 
	is highly unlikely, because it would require attaining high inner product 
	with a direction $u_j$ for $j>\mprog_i$ about which we have very little 
	information when $d$ is sufficiently large. 
	
	To make this intuition rigorous, we apply the union bound to the failure 
	probability, giving
	\begin{equation}\label{eq:prog-ub-1}
	\P\prn[\Bigg]{\bigcup_{i\le N}\{\mprog_i > i\}} 
	=\P\prn[\Bigg]{\bigcup_{i\le N}\{\mprog_i > i\andd p_{i-1}<i\}} 
	\le \sum_{i=1}^N
	\P\prn*{\mprog_i > i\andd \mprog_{i-1} < i},
	\end{equation}
	with $\mprog_0 = 0$. 
	We further upper bound each summand as
	\begin{equation}\label{eq:prog-ub-2}
	\P\prn*{\mprog_i > i, \mprog_{i-1} < i}
	= \P \prn[\Bigg]{\bigcup_{j\ge i}\crl[\Big]{\abs{\inner{u_j}{x_i}} > r\andd
			\mprog_{i-1}< 
			i} } \le (d-i+1) \cdot \P \prn[\bigg]{ \abs{\inner{u_i}{x_i}} > 
		r\andd\mprog_{i-1} < i},
	\end{equation}
	where the last step uses a union bound and the exchangeablility of 
	$u_{i},u_{i+1},\ldots,u_d$ under the event $\mprog_{i-1} 
	< i$. Note that the event $\mprog_{i-1} 
	< i$ implies that that $x_i$ depends on $\mU$ only through 
	$\mU\ind{<i}\defeq u_1, 
	\ldots, u_{i-1}$, as these vectors allow us to compute the oracle 
	responses to queries $x_1, \ldots, x_{i-1}$.\footnote{Applying Yao's 
		minimax principle~\citep{yao1977probabilistic} we implicitly 
		condition our 
		proof on the random coin tosses of $\alg$, which is tantamount to 
		assuming without loss of generality that $\alg$ is deterministic.} 
	Formally, we may write
	\begin{equation*}
	x_i = \algis(\mU\ind{<i}) \indic{\mprog_{i-1}<i} + 
	\algif(\mU)\indic{\mprog_{i-1}\ge i},
	\end{equation*}
	for two measurable functions $\algis : \R^{d \times (i-1)}\to\R^d$ and 
	$\algif: \R^{d\times d}\to \R^d$. Consequently, we have
	\begin{equation*}
	\P \prn[\bigg]{ \abs{\inner{u_i}{x_i}} > r\andd\mprog_{i-1} < i}
	=
	\P \prn[\bigg]{ \abs{\inner{u_i}{\algis(\mU\ind{<i})}} > 
		r\andd\mprog_{i-1} < i}
	\le 
	\P \prn[\bigg]{ \abs{\inner{u_i}{\algis(\mU\ind{<i})}} > r}.
	\end{equation*}
	Conditional on $\mU\ind{<i}$, the vector $u_i$ is uniformly 
	distributed in the $(d-i+1)$-dimensional space 
	$\mathrm{span}\{u_i,\ldots,u_d\}$. Therefore, standard 
	concentration inequalities on the 
	sphere~\citep[see][Lecture 8]{ball1997elementary} give
	\begin{equation*}
	\P \prn[\big]{ \abs{\inner{u_i}{\algis(\mU\ind{<i})}} > r \,\big\vert\, 
		\mU\ind{<i}} \le 2 
		\exp\crl*{-\frac{r^2}{2\norm{\algis(\mU\ind{<i})}^2}\cdot 
		(d-i+1)} \le \frac{\delta}{d^2},
	\end{equation*}
	where in the final step we substituted $\norm{\algis(\mU\ind{<i})}\le R$, and 
	our setting of $d$, which implies
	\begin{equation*}
	d-i+1 \ge d-N \ge \frac{20R^2}{r^2} \log \frac{20NR^2}{\delta \cdot 
		r^2} \ge 
	\frac{2R^2}{r^2} \log \frac{2d^2}{\delta}.
	\end{equation*}
	Substituting $ \P \prn*{ \abs{\inner{u_i}{x_i}} > r\andd\mprog_{i-1} < i} 
	\le \frac{\delta}{d^2}$ 
	into the bounds~\eqref{eq:prog-ub-1} 
	and~\eqref{eq:prog-ub-2} concludes the proof. 
\end{proof}

\subsection{Proof of \Cref{lem:nemirovski-func}}
\nemifunc*
\begin{proof}
	To prove the first part, fix $\bx$,  $x\in\ballset$ and  $j > \prog(\bx)$. 
	We have
	for all 
	$x\in\ballset$ that $|x_k - \bx_k| \le r$ for all $k\in[d]$, and 
	therefore
	\begin{equation*}
	x_j - 4r \cdot j \stackrel{(i)}{\le} \bx_j +r - 4r \cdot j \stackrel{(ii)}{\le} 
	\bx_{\prog(\bx)} +3r - 4 r \cdot j \stackrel{(iii)}{\le} 
	x_{\prog(\bx)} +4r - 4r\cdot j \stackrel{(iv)}{\le} 
	x_{\prog(\bx)} - 4r\cdot \prog(\bx).
	\end{equation*}
	Transitions $(i)$ and $(iii)$ above are due to $\norm{x-\bx}\le r$; 
	transition $(ii)$ is due to the definition~\eqref{eq:prog-def} of $\prog$, 
	which implies $|\bx_{\prog(\bx)}| \le r$ and $|\bx_j|\le r$; and $(iv)$ is 
	due to $j>\prog(\bx)$. Consequently, we have 
	\[
	f_{N,r}(x) = \max_{i\in[\prog(\bx)]} \{x_i - 4r\cdot i\}~~\mbox{for 
		all}~~x\in\ballset.
	\]
	Similarly, we can use $|\bx_{\prog(\bx)}| \le r$ and $j > \prog(\bx)$ to 
	conclude that
	\begin{equation*}
	0-4r\cdot j \le x_{\prog(\bx)} + r -4r\cdot j \le x_{\prog(\bx)} - 4r\cdot 
	\prog(\bx),
	\end{equation*}
	which means that $f_{N,r}(x) = \max_{i\in[\prog(\bx)]} \{x_i - 4r\cdot 
	i\}=f_{N,r}(x_1,\ldots,x_{\prog(\bx)}, 0, \ldots, 0)$, giving the robust 
	zero-chain property.
	
	The second property is well-known~\citep[see, 
	e.g.,][]{bubeck2019complexity}, but 
	we show it here for 
	completeness. Consider the point $\tilde{x}=-\frac{R}{\sqrt{N}}\ones \in 
	\domain$. Clearly, $\inf_{z\in \domain}f_{N,r}(z) \le f_{N,r}(\tilde{x}) = 
	-\frac{R}{\sqrt{N}}-4r$. Moreover, for any $x$ with $\prog(x) \le N$ we 
	have $f_{N,r}(x) \ge x_N - 4Nr \ge -(4N+1)r$. Combining these two 
	bounds yields $f_{N,r}(x) - \inf_{z\in 
		\domain}f_{N,r}(z)\ge 
	\frac{R}{\sqrt{N}} - (4N-3)r\ge \frac{R}{\sqrt{N}} - 4Nr$ as required.
	
	The final property follows from the fact that maximization preserves 
	convexity and Lipschitz constants.
\end{proof}

 \end{document}